\documentclass[reqno, 11pt]{amsart}%
\usepackage{amssymb}
\usepackage[nesting]{hyperref}
\usepackage[pdftex]{graphicx}
\usepackage{listings}
\usepackage{multirow}
\usepackage{placeins}
\usepackage{color}
\usepackage{subfigure}
\usepackage{lscape}
\usepackage{dsfont}
\usepackage{amsmath}
\usepackage{amsfonts}
\usepackage{tikz}
\usepackage{verbatim}
\usepackage{accents} 

\newcommand{\BlackBoxes}{\global\overfullrule5pt}

\BlackBoxes

\newcommand{\R}{\mathbb{R}}
\newcommand{\N}{\mathbb{N}}

\newcommand{\Pop}{\mathbb{P}}
\newcommand{\Q}{\mathbb{Q}}

\newcommand{\A}{\mathcal{A}}
\newcommand{\B}{\mathcal{B}}
\newcommand{\BB}{\mathbb{B}}

\newcommand{\E}{\mathbb{E}}

\renewcommand{\H}{\mathcal{H}}

\newcommand{\M}{\mathcal{M}}

\renewcommand{\P}{\mathbb{P}}
\newcommand{\Pc}{\mathcal{P}}

\newcommand{\Qc}{\mathcal{Q}}
\newcommand{\Qf}{\mathfrak{Q}}

\newcommand{\T}{\mathcal{T}}
\newcommand{\U}{\mathcal{U}}

\newcommand{\Z}{\mathcal{Z}}
\newcommand{\Zf}{\mathfrak{Z}}

\newcommand{\q}{F^{-1}}

\DeclareMathOperator{\argmax}{argmax}

\DeclareMathOperator{\conv}{conv}
\DeclareMathOperator{\cconv}{\overline{conv}}

\DeclareMathOperator{\dif}{d}

\DeclareMathOperator{\esssup}{ess\,sup}

\newcommand{\ubar}[1]{\underaccent{\bar}{#1}}

\newcommand{\wt}{\widetilde}

\newtheorem{theorem}{Theorem}
\newtheorem{corollary}[theorem]{Corollary}
\newtheorem{lemma}[theorem]{Lemma}
\newtheorem{proposition}[theorem]{Proposition}
\theoremstyle{definition}
\newtheorem{example}[theorem]{Example}
\newtheorem{remark}[theorem]{Remark}
\newtheorem{definition}[theorem]{Definition}
\newtheorem{assumptions}[theorem]{Assumption}

\numberwithin{equation}{section} \numberwithin{theorem}{section}
\def\0{\kern0pt\-\nobreak\hskip0pt\relax}

\makeatletter
\AtBeginDocument{ \def\@serieslogo{ \vbox to\headheight{ \parindent\z@ \fontsize{6}{7\p@}\selectfont
\vss}}}

\def\makeoverbar#1#2#3#4#5#6#7{ \setbox0=\hbox{$\m@th#2\mkern#5mu{{}#3{}}\mkern#6mu$} \setbox1=\null \dimen@=#4\fontdimen8#13 \dimen@=3.5\dimen@
\advance\dimen@ by \ht0 \dimen@=-#7\dimen@ \advance\dimen@ by \wd0
\ht1=\ht0 \dp1=\dp0 \wd1=\dimen@
\dimen@=\fontdimen8#13 \fontdimen8#13=#4\fontdimen8#13
\rlap{\hbox to \wd0{$\m@th\hss#2{\overline{\box1}}\mkern#5mu$}}
\fontdimen8#13=\dimen@}
\def\mylabel#1#2{{\def\@currentlabel{#2}\label{#1}}}
\makeatother
\begin{document}
\title[ Distributionally Robust Markov Decision Processes ]{Distributionally Robust  Markov Decision Processes and their Connection to Risk Measures}
\author[N. \smash{B\"auerle}]{Nicole B\"auerle${}^*$}
\address[N. B\"auerle]{Department of Mathematics,
Karlsruhe Institute of Technology (KIT), D-76128 Karlsruhe, Germany}

\email{\href{mailto:nicole.baeuerle@kit.edu}{nicole.baeuerle@kit.edu}}

\author[A. \smash{Glauner}]{Alexander Glauner${}^*$}
\address[A. Glauner]{Department of Mathematics,
Karlsruhe Institute of Technology (KIT), D-76128 Karlsruhe, Germany}

\email{\href{mailto:alexander.glauner@kit.edu} {alexander.glauner@kit.edu}}


\begin{abstract}
We consider robust Markov Decision Processes with Borel state and action spaces, unbounded cost and finite time horizon. Our formulation leads to a Stackelberg game against nature. Under integrability, continuity and compactness assumptions we derive a robust cost iteration for a fixed policy of the decision maker and a value iteration for the robust optimization problem. Moreover, we show the existence of deterministic optimal policies for both players. This is in contrast to classical zero-sum games. In case the state space is the real line we show under some convexity assumptions that the interchange of supremum and infimum is possible with the help of Sion's minimax Theorem.  Further, we consider the problem with special ambiguity sets. In particular we are able to derive some cases  where the robust optimization problem coincides with the minimization of a coherent risk measure. In the final section we discuss two applications: A robust LQ problem and a robust problem for managing regenerative energy.

\end{abstract}
\maketitle


\makeatletter \providecommand\@dotsep{5} \makeatother



\vspace{0.5cm}
\begin{minipage}{14cm}
{\small
\begin{description}
\item[\rm \textsc{ Key words}]
{\small Robust Markov Decision Process; Dynamic Games; Minimax Theorem; Risk Measures}
\item[\rm \textsc{AMS subject classifications}] 
{\small 90C40, 90C17, 91G70}
\end{description}
}
\end{minipage}

\section{Introduction}

Markov Decision Processes (MDPs) are a well-established tool to model and solve sequential decision making under stochastic perturbations. In the standard theory it is assumed that all parameters and distributions are known or can be estimated with a certain precision. However, using the so-derived 'optimal' policies in a system where the true parameters or distributions deviate, may lead to a significant degeneration of the performance. In order to cope with this problem there are different approaches in the literature.

The first approach which is typically used when parameters are unknown is the so-called {\em Bayesian approach}. In this setting a prior distribution for the model parameters is assumed and additional information which is revealed while the process evolves, is used to update the beliefs. Hence this approach allows that parameters can be learned. It is very popular in engineering applications. Introductions can e.g. be found in \cite{hernandez2012adaptive} and \cite{BaeuerleRieder2011}. In this paper we are not going to pursue  this stream of literature.

A second approach is the so-called {\em robust approach}. Here it is assumed that instead of having one particular transition law, we are faced with a whole family of laws which are possible for the transition. In the literature this is referred to as \emph{model ambiguity}.  One way of dealing with this ambiguity is the \emph{worst case approach}, where the controller selects a policy which is optimal with respect to the most adverse transition law in each scenario. This setting can also be interpreted as a dynamic game with nature as the controller's opponent. 
The worst case approach is empirically justified by the so-called \emph{Ellsberg Paradox}. The experiment suggested by \cite{Ellsberg1961} has shown that agents tend to be ambiguity averse. In the sequel, axiomatic approaches to model  risk and ambiguity attitude  have  appeared, see e.g. \cite{gilboa1989maxmin,maccheroni2006ambiguity}.
\cite{EpsteinSchneider2003} investigated the question whether ambiguity aversion can be incorporated in an axiomatic model of intertemporal utility. The representation of the preferences turned out to be some worst case expected utility, i.e. the minimal expected utility over an appropriate set of probability measures. This set of probability measures needs to satisfy some rectangularity condition for the utility to have a recursive structure and therefore being time consistent. 
The rectangularity property has been taken up by \cite{Iyengar2005} as a key assumption for being able to derive a Bellman equation for a robust MDP with countable state and action spaces. Contemporaneously, \cite{NilimGhaoui2005} reached similar findings, however limited to finite state and action spaces. 

\cite{wiesemann2013robust} have considered robust MDP beyond the rectangularity condition. Based on observed histories, they derive a confidence region that contains the unknown parameters with a prespecified probability and determine a policy that attains the highest worst-case performance over this confidence region. A similar approach has been taken in \cite{xu2010distributionally} where nested uncertainty sets for transition laws are given which correspond to confidence sets. The optimization is then based on the expected performance of policies under the (respective) most adversarial distribution. This approach lies between the Bayesian and the robust approach since the decision maker uses prior information without a Bayesian interpretation. All these analyses are restricted to finite state and action spaces. A similar but different  approach has been considered in \cite{bauerle2019markov} where parameter ambiguity is combined with Bayesian learning methods. Here the authors deal with arbitrary Borel state and action spaces.

In our paper we will generalize the results of \cite{Iyengar2005} to a model with Borel spaces and unbounded cost function. We consider finite horizon expected cost problems with a given transition law under ambiguity concerning the distribution of the disturbance variables. The problem is to minimize the worst expected cost. This leads to a Stackelberg game against nature.  In order to deal with the arising measurability issues which impose a major technical difficulty compared to the countable space situation, we borrow from the dynamic game setup in \cite{Gonzales2002} and \cite{JaskiewiczNowak2011}. 
The major difference of our contribution compared to these two works is the design of the distributional ambiguity. We replace the topology of weak convergence on the ambiguity set by the weak* topology $\sigma(L^q,L^p)$ in order to obtain connections to recursive risk measures in Section \ref{sec:special_sets}. Moreover, we rigorously derive a Bellman equation.  Note that \cite{jaskiewicz2014robust} treats another robust MDP  setup with Borel state and action spaces, however deals with the average control problem. 
Further, our model allows for rather general ambiguity sets for the disturbance distribution. 
Under additional technical assumptions our model also comprises the setting of a decreasing confidence regions for distributions. Moreover, we discuss in detail sufficient conditions for the interchange of supremum and infimum. We provide a counterexample which shows that this interchange is not always possible, in contrast to \cite{NilimGhaoui2005}, \cite{wiesemann2013robust}. This counterexample also highlights the difference to classical two-person zero-sum games. Further, we are able to derive a connection to the optimization of risk measures in MDP frameworks. In case the ambiguity sets are chosen in a specific way the robust optimization problem coincides with the optimization of a risk measure applied to the total cost.

The outline of the paper is as follows: In the next section we introduce our basic model which is a game against nature concerning expected cost with a finite time horizon. Section \ref{sec:abstract_robust_finite} contains the first main results. Under integrability, continuity and compactness assumptions we derive a robust cost iteration for a fixed policy of the decision maker and a value iteration for the robust optimization problem. Moreover, we show the existence of optimal deterministic policies for both players. This is in contrast to classical zero-sum games where we usually obtain optimal randomized policies.
In Section \ref{sec:abstract_robust_real} we consider the real line as state space which allows for slightly different model assumptions. Then in Section \ref{sec:abstract_robust_minimax} we discuss (with the real line as state space) when supremum and infimum can be interchanged in the solution of the problem. Under some convexity assumptions this can be achieved with the help of Sion's minimax Theorem \cite{Sion1958}. Being able to interchange supremum and infimum sometimes simplifies the solution of the problem. In Section \ref{sec:special_sets} we consider the problem with special ambiguity sets. In particular we are able to derive some cases which can be solved straightforward and situations where the robust optimization problem coincides with the minimization of a coherent risk measure. In the final section we discuss two applications: A robust LQ problem and  a robust problem for managing regenerative energy.

\section{The Markov Decision Model}\label{sec:abstract_cost_model}
We consider the following standard Markov Decision Process with general Borel state and action spaces and  restrict ourselves  to a model with \emph{finite planning horizon} $N \in \N$. Results for an infinite planning horizon can be found in \cite{glauner20}.  The \emph{state space} $E$ is a Borel space with Borel $\sigma$-algebra $\B(E)$ and the \emph{action space} $A$ is a Borel space with Borel $\sigma$-Algebra $\B(A)$. The possible state-action combinations at time $n$ form a measurable subset $D_n$  of $E \times A$ such that $D_n$ contains the graph of a measurable mapping $E \to A$. The $x$-section of $D_n$,  
	\[ D_n(x) = \{ a \in A: (x,a) \in D_n \}, \]
	is the set of admissible actions in state $x \in E$ at time $n$. The sets $D_n(x)$ are non-empty by assumption. We assume that the dynamics of the MDP are given by measurable \emph{transition functions} $T_n:D_n \times \Z \to E$ and depend on 
 \emph{disturbances} $Z_1,\dots,Z_N$ which are independent random elements on a common probability space $(\Omega,\A,\P)$ with values in a measurable space $(\Z, \Zf)$. When the current state is $x_n$ the controller chooses action $a_n\in D_n(x_n)$ and $z_{n+1}$ is the realization of $Z_{n+1}$, then the next state is given by
	\[ x_{n+1} = T_n(x_n,a_n,z_{n+1}). \] The \emph{one-stage cost function} $c_n:D_n\times E \to \R$ gives the cost $c_n(x,a,x')$  for choosing action $a$ if the system is in state $x$ at time $n$ and the next state is $x'$.
The \emph{terminal cost function} $c_N: E \to \R$ gives the cost $c_N(x)$ if the system terminates in state $x$.

The model data is supposed to have the following continuity and compactness properties. 
\begin{assumptions}\phantomsection\label{ass:continuity_compactness}
	\begin{enumerate}
		\item[(i)]  $D_n(x)$ are compact and  $E \ni x \mapsto D_n(x)$ is upper semicontinuous  for $n=0,\dots,N-1$, i.e. if $x_k\to x$ and $a_k\in D_n(x_k)$, $k\in\N$, then $(a_k)$ has an accumulation point in $D_n(x)$.
		\item[(ii)] $T_n(x,a,z)$ is continuous in $(x,a)$ for $z\in \Z$ and $n=0,\dots,N-1$.
		\item[(iii)]  $c_n,  n=0,\dots,N-1$, as well as the terminal cost function $c_N$ are lower semicontinuous. 
	\end{enumerate}
\end{assumptions}


For $n \in \N_0$ we denote by $\H_n$ the set of \emph{feasible histories} of the decision process up to time $n$
\begin{align*}
h_n = \begin{cases}
x_0, & \text{if } n=0,\\
(x_0,a_0,x_1, \dots, x_n), & \text{if } n \geq 1,
\end{cases}
\end{align*}
where $a_k \in D_k(x_k)$ for $k \in \N_0$. In order for the controller's decisions to be implementable, they must be based on the information available at the time of decision making, i.e. be functions of the history of the surplus process. 

\begin{definition}
	\begin{itemize}
		\item[(i)] A \emph{randomized policy} is a sequence $\pi = (\pi_0,\pi_1,\dots,\pi_{N-1})$ of stochastic kernels $\pi_n$ from $\mathcal{H}_n$ to the action space $A$  satisfying the constraint 
		\[ \pi_n(D_n(x_n)|h_n) = 1, \qquad h_n \in \mathcal{H}_n. \] 
		\item[(ii)] A measurable mapping $d_n: \mathcal{H}_n \to A$ with $d_n(h_n) \in D_n(x_n)$ for every $h_n \in \mathcal{H}_n$ is called (deterministic) \emph{decision rule} at time $n$.  $\pi=(d_0, d_1, \dots,d_{N-1})$ is called (deterministic) \emph{ policy}.
		\item[(iii)] A decision rule at time $n$ is called \emph{Markov} if it  depends on the current state only, i.e.\ $d_n(h_n)=d_n(x_n)$ for all $h_n \in \mathcal{H}_n$. If all decision rules are Markov, the deterministic ($N$-stage) policy is called \emph{Markov}.
	\end{itemize}
\end{definition}
For convenience, deterministic policies may simply be referred to as policy. With $\Pi^R \supseteq \Pi \supseteq \Pi^M $ we denote the sets of all randomized policies, deterministic policies and Markov policies. The first inclusion is by identifying deterministic decision rules $d_n$ with the corresponding Dirac kernels
\[ \pi_n(\cdot|h_n):= \delta_{d_n(h_n)}(\cdot), \qquad h_n \in \H_n. \]
 A feasible policy always exists since $D_n$ contains the graph of a measurable mapping.

 Due to the independence of the disturbances we may without loss of generality assume that the probability space has a product structure
\[ (\Omega,\A,\P) = \bigotimes_{n=1}^\N (\Omega_n,\A_n,\P_n). \]

We take $(\Omega,\A,\P)$ as the canonical construction, i.e.
\[ (\Omega_n,\A_n,\P_n) =  \left(\mathcal{Z},\mathfrak{Z},\P^{Z_n}\right) \qquad \text{and} \qquad Z_n(\bar \omega) = \omega_n, \quad  \bar \omega= (\omega_1,\dots,\omega_N) \in \Omega \]
for all $n =1,\dots,N$. We denote by $(X_n), (A_n)$ the random state and action processes and define $H_n:=(X_0,A_0,\ldots,X_n)$.  In the sequel, we will require $\P_n$ to be separable. Additionally, we will assume for some results that $\left(\Omega_n,\A_n,\P_n \right)$ is atomless in order to support a generalized distributional transform. 

Let $n \in \{0,\dots,N-1\}$ be a stage of the decision process. Due to the product structure of $(\Omega,\A,\P)$ the transition kernel is given by
\begin{align}\label{eq:abstract_transition_kernel_product_space}
Q_n(B|x,a)= \int 1_B\big(T_n(x,a,z_{n+1})\big) \P_{n+1}(\dif z_{n+1}), \quad B \in \B(E), \ (x,a) \in D_n.
\end{align}
We assume now that there is some uncertainty about $\Pop_n$, e.g. because it cannot be estimated properly. Moreover, the decision maker is very risk averse and tries to minimize the expected cost on a worst case basis. Thus, we denote by $\M(\Omega_n,\A_n,\P_n )$ the set of probability measures on $(\Omega_n,\A_n)$ which are absolutely continuous with respect to $\P_n$ and define for $q\in (1,\infty]$
\[ \M^q(\Omega_n,\A_n,\P_n) = \left\{ \Q \in\M(\Omega_n,\A_n,\P_n): \frac{\dif \Q}{\dif \P_n} \in L^q(\Omega_n,\A_n,\P_n) \right\}. \] 
Henceforth, we fix a non-empty subset $\Qc_n \subseteq \M^q(\Omega_n,\A_n,\P_n)$ which is referred to as \emph{ambiguity set} at stage $n$. This can be seen as the set of probability measures which may reflect the  law of motion.  Due to absolute continuity, we can identify $\Qc_n$ with the set of corresponding densities w.r.t.\ $\P_n$
\begin{equation}\label{eq:densities}
\Qc_n^d= \left\{ \frac{\dif \Q}{\dif \P_n} \in L^q(\Omega_n,\A_n,\P_n): \Q \in \Qc_n \right\}.
\end{equation}
Accordingly, we view $\Qc_n$ as a subset of $L^q(\Omega_n,\A_n,\P_n)$ and endow it with the trace topolgy of the weak* topolgy $\sigma(L^q,L^p)$ on $L^q(\Omega_n,\A_n,\P_n)$ where $\frac1p+\frac1q=1$. The weak* topology in turn induces a Borel $\sigma$-algebra on $\Qc_n$ making it a measurable space. We obtain the following result (for a proof see the Appendix).

\begin{lemma}\label{thm:ambiguity_sep_metrizable}
	Let the ambiguity set be norm-bounded and the probability measure $\P_n$ on $(\Omega_n,\A_n)$ be separable. Then $\Qc_n$ endowed with the weak* topology $\sigma(L^q,L^p)$ is a separable metrizable space. If $\Qc_n$ is additionally weak* closed, it is even a compact Borel space. 
\end{lemma}

In our cost model we  allow for any norm-bounded ambiguity set $\Qc_n \subseteq \M^q(\Omega_n,\A_n,\P_n)$. For applications, a meaningful way of choosing $\Qc_n$ (within a norm bound) is to take all probability measures in $\M^q(\Omega_n,\A_n,\P_n)$  which are  close to $\P_n$ in a certain metric like e.g. the \emph{Wasserstein metric} (see e.g. \cite{yang2017convex}).  In our setting, that requires absolute continuity, the \emph{Kullback–-Leibler divergence} could be a natural choice.

The controller only knows that the transition kernel \eqref{eq:abstract_transition_kernel_product_space} at each stage is defined by some $\Q \in \Qc_{n+1}$ instead of $\P_{n+1}$ but not which one exactly. We assume here that the controller faces a dynamic game against nature. This means that nature reacts to the controller's action $a_n$ in scenario $h_n \in \H_n$ with a decision rule $\gamma_n: \H_n\times A \to \Qc_{n+1}$ where $a_n\in D_n(x_n)$. A \emph{policy of nature} is a sequence of such decision rules $\gamma = (\gamma_0,\dots,\gamma_{N-1})$. Let $\Gamma$ be the set of all policies of nature. Since nature is an unobserved theoretical opponent of the controller, her actions are not considered to be part of the history of the Markov Decision Process. A \emph{Markov decision rule of nature} at time $n$ is a measurable mapping $\gamma_n:D_n \to \Qc_{n+1}$ and a \emph{Markov policy of nature} is a sequence $\gamma = (\gamma_0,\dots,\gamma_{N-1})$ of such decision rules. The set of Markov policies of nature is denoted by $\Gamma^M$. Thus we are faced with a {\em Stackelberg game} where the controller is the mover and nature is the follower.

\begin{lemma}\label{thm:abstract_robust_kernel}
	For $n= 0,\dots,N-1$ a decision rule $\gamma_n:\H_n\times A \to \Qc_{n+1}$ induces a stochastic kernel from $\H_n\times A$ to $\Omega_{n+1}$ by
	\[ \gamma_n(B|h_n,a_n):= \gamma_n(h_n,a_n)(B), \qquad B \in \A_{n+1}, \ (h_n,a_n)\in \H_n\times A. \]
\end{lemma}

For a proof of the lemma see the Appendix. In the sequel, it will be clear from the context where we refer to $\gamma_n$ as a decision rule or as a stochastic kernel. 

The probability measure $\gamma_n(\cdot|h_n,a_n)$, which is unknown for the controller, now takes the role of $\P_{n+1}$ in defining the transition kernel of the Markov Decision Process in \eqref{eq:abstract_transition_kernel_product_space}. Let 
\begin{align}\label{eq:abstract_transition_kernel_by_nature}
	 Q_n^\gamma(B|h_n,a_n)=  \int 1_B\big(T(x_n,a_n,z_{n+1})\big) \gamma_n(\dif z_{n+1}|h_n,a_n), \; B \in \B(E), \ h_n \in \H_n, a_n\in D_n(x_n).
\end{align}
As in the case without ambiguity, the Theorem of Ionescu-Tulcea  yields that each initial state $x \in E$ and pair of policies of the controller and nature $(\pi,\gamma) \in \Pi^R \times \Gamma$ induce a unique law of motion 
\begin{align}\label{eq:abstract_robust_law_of_motion}
\Q^{\pi\gamma}_x := \delta_x \otimes \pi_0 \otimes Q_0^\gamma \otimes \pi_1 \otimes Q_1^\gamma \otimes \dots \otimes \pi_{N-1} \otimes Q_{N-1}^\gamma
\end{align}
on ${\H}_N$ satisfying
\begin{align*}
\Q^{\pi \gamma}_x(X_0 \in B) &= \delta_x(B),\\
\Q^{\pi \gamma}_x(A_n \in C|H_n=h_n) &= \pi_n(C|h_n),\\
\Q^{\pi \gamma}_x(X_{n+1} \in B|(H_n, A_n)=(h_n,a_n)) &= Q_n^\gamma(B|h_n,a_n)
\end{align*}
for all $B \in \B(E)$ and $C \in \B(A)$. In the usual way, we denote with $\E_{x}^{\pi\gamma}$ the expectation operator with respect to $\Q^{\pi\gamma}_x$ and with $\E_{nh_n}^{\pi\gamma}$ or $\E_{nx}^{\pi\gamma}$ the respective conditional expectation given $H_n=h_n$ or $X_n=x$.

The value of a policy pair $(\pi,\gamma) \in \Pi^R \times \Gamma$ at time $n=0,\dots, N-1$ is defined as
\begin{align}\label{eq:abstract_robust_policy_value}
\begin{aligned}
V_{N \pi \gamma}(h_N) &=  c_N(x_N), & h_N \in \H_N,\\
V_{n \pi \gamma}(h_n) &=  \E_{n h_n}^{\pi \gamma} \left[ \sum_{k=n}^{N-1} c_k(X_k,A_k,X_{k+1}) + c_N(X_N) \right], & h_n \in \H_n.
\end{aligned}
\end{align}	
Since the controller is unaware which probability measure in the ambiguity set determines the transition law in each scenario, he tries to minimize the expected cost under the assumption to be confronted with the most adverse probability measure. The value functions are thus given by
\begin{align*}
V_n(h_n) = \inf_{\pi \in \Pi^R} \sup_{\gamma \in \Gamma} \ V_{n\pi \gamma}(h_n), \qquad h_n \in \mathcal{H}_n,
\end{align*}
and the optimization objective is 
\begin{align}\label{eq:opt_crit_abstract_robust_finite}
V_0(x)=\inf_{\pi \in \Pi^R} \sup_{\gamma \in \Gamma} \ V_{0\pi\gamma}(x), \qquad x \in E.
\end{align}
In game-theoretic terminology this is the \emph{upper value of a dynamic zero-sum game}. If nature were to act first, i.e. if infimum and supremum were interchanged, one would obtain the game's \emph{lower value}. If the two values agree and the infima/suprema are attained, the game has a \emph{Nash equilibrium}, see also Section \ref{sec:abstract_robust_minimax}. But note here that players are asymmetric in information in our setting.

\begin{remark}
\cite{Iyengar2005} does not model nature to make active decisions, but instead defines the set of all possible laws of motion. When each law of motion is of the form \eqref{eq:abstract_robust_law_of_motion}, he calls the set \emph{rectangular}. 
Our approach with active decisions of nature based on  \cite{Gonzales2002} and \cite{JaskiewiczNowak2011} is needed to construct Markov kernels as in Lemma \ref{thm:abstract_robust_kernel} with probability measures from a given ambiguity set. When state and action spaces are countable as in \cite{Iyengar2005} the technical problem of measurability does not arise and one can directly construct an ambiguous law of motion by simply multiplying (conditional) probabilities. The rectangularity property is satisfied in our setting.
\end{remark}

Our model feature that there is no ambiguity in the transition functions is justified in many applications. Typically, transition functions describe a technical process or economic calculation which is known ex-ante and does not have to be estimated. The same applies to the cost function.

\section{Value Iteration and Optimal Policies}\label{sec:abstract_robust_finite}
In order to have well-defined value functions we need some integrability conditions. 
\begin{assumptions} \phantomsection \label{ass:abstract_robust_finite}
	\begin{itemize}
		\item[(i)] There exist $\alpha, \ubar \epsilon, \bar \epsilon \geq 0$ with $\ubar \epsilon +\bar \epsilon =1, \alpha\neq 1$ and measurable functions $\ubar b: E \to (-\infty, -\ubar \epsilon]$ and $\bar b:E \to [\bar \epsilon,\infty)$ such that it holds for all $n=0,\dots,N-1$, $\Q \in \Qc_{n+1}$ and $(x,a)\in D_n$
		\begin{align*}
		\E^\Q\left[-c_n^-(x,a,T_n(x,a,Z_{n+1})) \right] &\geq \ubar b(x), &
		\E^\Q\left[ \ubar b(T_n(x,a,Z_{n+1}))  \right] &\geq \alpha \ubar b(x),\\
		\E^\Q\left[c_n^+(x,a,T_n(x,a,Z_{n+1})) \right] &\leq \bar b(x), &
		\E^\Q\left[ \bar b(T_n(x,a,Z_{n+1}))  \right] &\leq \alpha \bar b(x).
		\end{align*}
		Furthermore, it holds $\ubar b(x) \leq c_N(x) \leq \bar b(x)$ for all $x \in E$.
		\item[(ii)] We define $b:E \to [1,\infty), \ b(x) :=\bar b(x) - \ubar b(x)$. For all $n=0,\dots,N-1$ and $(\bar x, \bar a) \in D_n$ there exists an $\epsilon >0$ and measurable functions $\Theta_{n,1}^{\bar x, \bar a}, \Theta_{n,2}^{\bar x, \bar a}: \Z \to \R_+$ such that $\Theta_{n,1}^{\bar x, \bar a}(Z_{n+1}), \Theta_{n,2}^{\bar x, \bar a}(Z_{n+1}) \in L^p(\Omega_{n+1},\A_{n+1},\P_{n+1})$ and
		\begin{align*}
		|c_n(x,a,T_n(x,a,z))| \leq \Theta_{n,1}^{\bar x, \bar a}(z), \qquad \qquad  b(T_n(x,a,z)) \leq \Theta_{n,2}^{\bar x, \bar a}(z)
		\end{align*}
		for all  $z \in \Z$ and $(x,a) \in B_\epsilon(\bar x, \bar a) \cap D_n$. Here, $B_\epsilon(\bar x, \bar a)$ is the closed ball around $(\bar x, \bar a)$ w.r.t.\ an arbitrary product metric on $E\times A$.
		\item[(iii)] The ambiguity sets $\Qc_n$ are norm bounded, i.e.\ there exists $K \in [1,\infty)$ such that
		\[ \E \left| \frac{\dif \Q}{\dif \P_n} \right|^q \leq K  \]
		for all $n=1,\dots,N$ and $\Q \in \Qc_n$.
	\end{itemize}
\end{assumptions} 

\begin{remark}\phantomsection\label{rem:abstract_robust_finite_assumptions}
	\begin{enumerate}
		\item $\ubar b, \bar b$ are called \emph{lower} and \emph{upper bounding function}, respectively, while $b$ is referred to as \emph{bounding function.} As the absolute value is the sum of positive and negative part, $b$ satisfies for all $n=0,\dots,N-1$, $\Q \in \Qc_n$ and $ (x,a) \in D_n$:
		$$ \E^\Q\left[|c_n(x,a,T_n(x,a,Z_{n+1}))| \right] \leq  b(x) \qquad \text{and} \qquad \E^\Q\left[  |b(T_n(x,a,Z_{n+1}))|  \right] \leq \alpha  b(x)$$
		\item Assumptions \ref{ass:abstract_robust_finite} (i) and (ii) are satisfied with $-\ubar b=\bar b=constant>0$ if the cost functions are bounded.
		\item If $p=1$ and hence $q=\infty$, it is technically sufficient if part (ii) of Assumption \ref{ass:abstract_robust_finite} holds under the reference probability measure $\P_n$. Using H\"older's inequality and part (iii) we get for every $\Q \in \Qc_{n+1}$
		\begin{align*}
		\E^\Q\left[-c_n^-(x,a,T_n(x,a,Z_{n+1})) \right] &\geq \E\left[-c_n^-(x,a,T_n(x,a,Z_{n+1})) \right] \esssup \frac{\dif \Q}{\dif \P_{n+1}} \geq K  \ubar b(x),\\
		\E^\Q\left[ \ubar b(T_n(x,a,Z_{n+1}))  \right] &\geq \E\left[ \ubar b(T_n(x,a,Z_{n+1}))  \right] \esssup \frac{\dif \Q}{\dif \P_{n+1}} \geq  \alpha K \ubar b(x)
		\end{align*}
		and analogous results for the upper bounding function. I.e.\ one simply has to replace $\alpha$ by $K \alpha$. However, the factor $K \alpha$ may be unnecessarily crude. 
	\end{enumerate}
\end{remark}

The next lemma shows that due to Assumption \ref{ass:abstract_robust_finite} (i) the value \eqref{eq:abstract_robust_policy_value} of a policy pair $(\pi,\gamma) \in \Pi^R \times \Gamma$ is well-defined at all stages $n=0,\dots,N$. One can see that the existence of either a lower or an upper bounding function is sufficient for the policy value to be well-defined. However, for the existence of an optimal policy pair we will need the integral to exist with finite value and therefore require both a lower and an upper bounding function.

\begin{lemma}\label{thm:abstract_robust_bounds}
	Under Assumption \ref{ass:abstract_robust_finite} it holds for all policy pairs $(\pi,\gamma) \in \Pi^R \times \Gamma$, time points $n=0,\dots, N$ and histories $h_n \in \mathcal{H}_n$
  \[   \frac{1-\alpha^{N+1-n}}{1-\alpha}  \ubar b(x_n) \le V_{n \pi \gamma}(h_n)  \leq \frac{1-\alpha^{N+1-n}}{1-\alpha}  \bar b(x_n). \]
\end{lemma}

\begin{proof}
	We only prove the first inequality.  The second inequality  is analogous.  We use that
	\begin{align*}
		 V_{n \pi \gamma}(h_n)  \ge & \sum_{k=n}^{N-1} \E_{n h_n}^{\pi \gamma}\left[ -c_k^-(X_k,A_k,X_{k+1}) \right]+ \E_{N h_N}^{\pi \gamma}\left[-c_N^-(X_N)\right]
	\end{align*}
	and consider the summands individually. We have $\E_{N h_N}^{\pi \gamma}\left[-c_N^-(X_N)\right] \geq \E_{N h_N}^{\pi \gamma}\left[\ubar b(X_N)\right]$ by Assumption \ref{ass:abstract_robust_finite} (i). Since $\gamma_k$ is a mapping to $\Qc_{k+1}$ it follows from the first inequality of Assumption \ref{ass:abstract_robust_finite} (i) that
	\begin{align*}
	&\E_{n h_n}^{\pi \gamma}\left[ -c_k^-(X_k,A_k,X_{k+1}) \right]= \int \E_{k h_k}^{\pi \gamma}\left[ -c_k^-(X_k,A_k,X_{k+1}) \right] \Q_x^{\pi\gamma}(\dif h_{k}|H_n=h_n)\\
	&= \iiint  -c_k^-\big(x_k,a_k,T_k(x_k,a_k,z_{k+1})\big) \, \gamma_k(\dif z_{k+1}|h_k,a_k) \, \pi_k(\dif a_k|h_k) \ \Q_x^{\pi\gamma}(\dif h_{k}|H_n=h_n)\\
	&\geq \int \ubar b(x_k) \
	 \Q_x^{\pi\gamma}(\dif h_{k}|H_n=h_n)	= \E_{n h_n}^{\pi\gamma}\left[\ubar b(X_k)\right]
	\end{align*}
	for $k=n,\dots,N-1$. Now, the second inequality of Assumption \ref{ass:abstract_robust_finite} (i) yields for $k \geq n+1$
	\begin{align*}
	&\E_{n h_n}^{\pi\gamma}\left[\ubar b(X_k)\right]= \int \E_{k-1 h_{k-1}}\left[\ubar b(X_k)\right] \Q_x^{\pi\gamma}(\dif h_{k-1}|H_n=h_n)\\
	&= \iiint \! \ubar b\big(T_{k-1}(x_{k-1},a_{k-1},z_{k})\big) \, \gamma_{k-1}(\dif\! z_k|h_{k-1},a_{k-1})\, \pi_{k-1}(\dif\! a_{k-1}|h_{k-1}) \, \Q_x^{\pi\gamma}(\dif\! h_{k-1}|H_n\!=\!h_n)\\
	&\geq \alpha \int \ubar b(x_{k-1}) \Q_x^{\pi\gamma}(\dif h_{k-1}|H_n=h_n)= \alpha \E_{n h_{n}}^{\pi \gamma} \left[ \ubar b(X_{k-1}) \right].
	\end{align*}
	Iterating this argument we obtain
	\[ \E_{n h_n}^{\pi\gamma}\left[ -c_N^-(X_N) \right] \geq \alpha^{N-n}  \ubar b(x_n) \qquad \text{and} \qquad \E_{n h_n}^{\pi\gamma}\left[ -c_k^-(X_k,A_k,X_{k+1}) \right] \geq \alpha^{k-n} \ubar b(x_n). \]
	Finally, summation over $k$ yields the assertion.
\end{proof}

With the bounding function $b$ we define the function space
\[ \BB_b := \left\{ v: E \to \R \ \vert \ v \text{ measurable with } \lambda \in \R_+ \text{ s.t. } |v(x)| \leq \lambda \,  b(x) \text{ for all } x \in E \right\}. \]
Endowing $\BB_b$ with the weighted supremum norm
\[ \|v\|_b := \sup_{x \in E} \frac{|v(x)|}{b(x)}. \]
makes $(\BB_b, \|\cdot\|_b )$ a Banach space, cf. Proposition 7.2.1 in \cite{HernandezLasserre1999}.

Having ensured that the value functions are well-defined, we can now proceed to derive the cost iteration. To ease notation we introduce the following operators.

\begin{definition}
 For functions $v: \H_{n+1}\to \R $ s.t.\ $v(h_n,a_n,\cdot) \in \BB_b$ for all $h_n\in\H_n, a_n\in D_n(x_n)$ and $n=0,\ldots,N-1$ let
\begin{eqnarray*}
&&\T_{n\pi_n\gamma_n}v(h_n) :=\\
&& \iint c_n\big(x_n,a_n,T_n(x_n,a_n,z_{n+1})\big)+
 v\big(h_n, a_n,T_n(x_n,a_n,z_{n+1})\big) \, \gamma_n(\dif z_{n+1}|h_n, a_n) \, \pi_n(\dif a_n|h_n)\\
 &&\T_{n\pi_n}v(h_n) :=\\
 && \int  \sup_{\Q \in \Qc_{n+1}} \int c_n\big(x_n,a_n,T_n(x_n,a_n,z_{n+1})\big)+
 v\big(h_n, a_n,T_n(x_n,a_n,z_{n+1})\big) \, \Q(\dif z_{n+1}) \, \pi_n(\dif a_n|h_n)
\end{eqnarray*}
\end{definition}

Note that the operators are monotone in $v$.

\begin{proposition}\label{thm:abstract_robust_value_iteration}
	Under Assumption \ref{ass:abstract_robust_finite} the value of a policy pair $(\pi,\gamma) \in \Pi^R \times \Gamma$ can be calculated recursively for $n=0,\dots,N$ and $h_n \in \H_n$ as
	\begin{align*}
	V_{N\pi\gamma}(h_N) &= c_N(x_N),\\
	V_{n \pi\gamma}(h_n) &=  \T_{n\pi_n\gamma_n} V_{n+1\pi\gamma}(h_n).
	\end{align*}
\end{proposition}

\begin{proof}
	The proof is by backward induction. At time $N$ there is nothing to show. Now assume the assertion holds for $n+1$, then the tower property of conditional expectation yields for  $n$
	\begin{align*}
	&V_{n \pi \gamma}(h_n)	
	= \E_{n h_n}^{\pi \gamma} \left[ c_n(X_n,A_n,X_{n+1}) + \E_{n+1 h_n A_n X_{n+1}}^{\pi\gamma} \left[ \sum_{k=n+1}^{N-1} c_k(X_k,A_k,X_{k+1}) +c_N(X_N)\right] \right]\\
	&= \iint c_n(x_n,a_n,x') + \E_{n+1 h_n a_n x'}^{\pi\gamma} \left[ \sum_{k=n+1}^{N-1} c_k(X_k,A_k,X_{k+1}) +c_N(X_N)\right]  Q_n^\gamma(\dif x'|h_n, a_n) \, \pi_n(\dif a_n|h_n)\\
	&  =\T_{n\pi_n\gamma_n} V_{n+1\pi\gamma}(h_n)
	\end{align*} 
	for all $h_n \in \H_n$.
\end{proof}

Now, we evaluate a policy of the controller under the worst-case scenario regarding nature's reaction.  We define the \emph{robust value of a policy} $\pi \in \Pi^R$ at time $n=0,\dots, N-1$ as
\[ V_{n\pi}(h_n) = \sup_{\gamma \in \Gamma} V_{n\pi\gamma}(h_n), \quad h_n \in \mathcal{H}_n. \]
To minimize this quantity is the controller's optimization objective. For the robust policy value, a cost iteration holds, too. With regard to a policy of nature this is a Bellman equation given a fixed policy of the controller.

\begin{theorem}\label{thm:abstract_robust_nature_bellman}
	Let Assumptions \ref{ass:continuity_compactness}, \ref{ass:abstract_robust_finite} be satisfied.
	\begin{itemize}
		\item[a)] The robust value of a policy  $\pi \in \Pi^R $ is a measurable function of $h_n \in \H_n$ for $n=0,\dots,N-1$. It can be calculated recursively as
		\begin{alignat*}{3}
		V_{N \pi}(h_N) &= c_N(x_N),\\
		V_{n \pi}(h_n) &=  \T_{n\pi_n} V_{n+1\pi}(h_n)
		\end{alignat*}
		\item[b)] If the ambiguity sets $\Qc_{n+1}$ are weak* closed, there exist  maximizing decision rules $\gamma_n^*$ of nature for all $n$. 
		Each sequence of such decision rules $\gamma^*=(\gamma_1^*,\dots,\gamma_{N-1}^*)\in \Gamma$ is an optimal response of nature to the controller's policy, i.e. $ 	V_{n \pi} = V_{n \pi\gamma^*}$, $ n=0,\dots,N-1. $
	\end{itemize}
\end{theorem}

\begin{proof}
The proof is by backward induction. At time $N$ there is nothing to show. Now assume the assertion holds at time $n+1$, i.e. that $V_{n+1\pi}$ is measurable and that for every $\epsilon > 0$ there exists an $\frac\epsilon 2$-optimal strategy $\hat \gamma = (\hat \gamma_{n+1}, \dots, \hat \gamma_{N-1})$ of nature. By Proposition \ref{thm:abstract_robust_value_iteration} we have at  $n$
\begin{eqnarray}\nonumber
V_{n\pi}(h_n) &=&   \sup_{\gamma \in \Gamma} V_{n\pi\gamma}(h_n)=\sup_{\gamma \in \Gamma} \T_{n\pi_n\gamma_n}V_{n+1\pi\gamma}(h_n) \le  \sup_{\gamma \in \Gamma} \T_{n\pi_n\gamma_n}V_{n+1\pi}(h_n)\le  \T_{n\pi_n}V_{n+1\pi}(h_n)\\ 
&\le &\!\!\!  \T_{n\pi_n\hat \gamma_n}V_{n+1\pi}(h_n)+\frac{\varepsilon}{2} \le \T_{n\pi_n\hat \gamma_n}V_{n+1\pi \hat \gamma}(h_n)\! +\! \varepsilon =  V_{n\pi \hat \gamma}(h_n) \! +\! \epsilon		\leq V_{n\pi}(h_n) + \epsilon \label{eq:th3.6}
\end{eqnarray}
where  $\hat \gamma_n:\H_n\times A \to \Qc_{n+1}$ is a measurable $\frac\epsilon 2$-maximizer.

		Since $\epsilon > 0$ is arbitrary, equality holds. It remains to show the measurability of the outer integrand after the second inequality and the existence of an $\frac\epsilon 2$-maximizer. This follows from the optimal measurable selection theorem in  \cite{Rieder1978}:  To see this, first note that the function
		\[ f(h_n,a_n,\Q) = \int c_n\big(x_n,a_n,T_n(x_n,a_n,Z_{n+1})\big)+ V_{n+1\pi}\big(h_n,a_n,T_n(x_n,a_n,Z_{n+1})\big) \dif \Q,  \]
is jointly measurable due to
		Lemma 4.51 in \cite{AliprantisBorder2006}.
		Consequently,
		\[ \left\{ (h_n,a_n,\Q) \in \H_n\times A \times \Qc_{n+1}: f(h_n,a_n,\Q) \geq \eta \right\} \in \left\{S\times Q: S \in \B(\H_n\times A), \ Q \subseteq \Qc_{n+1} \right\}. \]
		for every $\eta \in \R$. The right hand side is a selection class. Obviously, it holds
		\[ \H_n\times A \times \Qc_{n+1} \in \left\{S\times Q: S \in \B(\H_n\times A), \ Q \subseteq \Qc_{n+1} \right\}. \]
		Now, Theorem 3.2 in  \cite{Rieder1978} yields that
		\[ \H_n\times A \ni (h_n,a_n) \mapsto \sup_{\Q \in \Qc_{n+1}} f(h_n,a_n,\Q) \] 
		is measurable and for every $\epsilon >0$ there exists an $\epsilon$-maximizer $\gamma_n:\H_n\times A \to \Qc_{n+1}$.

	For part b) we have to show that there exists not only a $\epsilon$-maximizer $\hat \gamma_n$ at \eqref{eq:th3.6} but a maximizer. This follows from Theorem 3.7 in \cite{Rieder1978}. The additional requirements are that $\Qc_{n+1}$ is a separable metrizable space, which holds by Lemma \ref{thm:ambiguity_sep_metrizable}, and that the set $\{ \Q \in \Qc_{n+1}: f(h_n,a_n,\Q) \geq \eta \}$
		is compact for every $\eta \in \R$ and $(h_n,a_n) \in \H_n\times A$. By assumption, $\Qc_{n+1}$ is weakly closed and therefore compact by Lemma \ref{thm:ambiguity_sep_metrizable}. Since due to Assumption \ref{ass:abstract_robust_finite} the integrand of $f$ is in $L^p$, the mapping $\Q \mapsto f(h_n,a_n,\Q)$ is continuous for every $(h_n,a_n) \in \H_n\times A$. Hence,
		$ \{ \Q \in \Qc_{n+1}: f(h_n,a_n,\Q) \geq \eta \}$
		is closed as the preimage of a closed set. Since closed subsets of compact sets are compact, the proof is complete. \qedhere
\end{proof}

So far we have only considered the case that the ambiguity set may depend on the time index but not on the state of the decision process. This covers many applications, e.g.\ the connection to risk measures (see  Section \ref{sec:special_sets}). Moreover, we can allow any norm bounded ambiguity sets as long as it is independent of the state using the optimal selection theorem by \cite{Rieder1978} in Theorem \ref{thm:abstract_robust_nature_bellman}. If the ambiguity set is weak* closed, the following generalization is possible.

\begin{corollary}
	For $n=0,\dots, N-1$ let $\Qc_n$ be weak* closed and 
		\[ D_n \ni (x,a) \mapsto \Qc_{n+1}(x,a) \subseteq \Qc_{n+1} \]
	be a non-empty and closed-valued  mapping giving the possible probability measures at time $n$ in state $x \in E$ if the controller chooses $a \in D_n(x)$. Then the assertion of Theorem \ref{thm:abstract_robust_nature_bellman} b) still holds.
\end{corollary}

\begin{proof}
	We have to show the existence of a measurable maximizer at \eqref{eq:th3.6}. The rest of the  proof is not affected. Since $\Qc_{n+1}$ is weak* closed, it is a compact Borel space by Lemma \ref{thm:ambiguity_sep_metrizable}. Consequently, the set-valued mapping $\Qc_{n+1}(\cdot)$ is compact-valued, as closed subsets of compact sets are compact. In the proof of the theorem it has been shown that the function $f(h_n,a_n,\Q)$ is jointly measurable and continuous in $\Q$. Hence, Proposition D.5 in \cite{HernandezLasserre1996} yields the existence of a measurable maximizer.
\end{proof}

State dependent ambiguity sets are a possibility to make the distributionally robust optimality criterion less conservative. E.g.\ they allow to incorporate learning about the unknown disturbance distribution. We refer the reader to \cite{Bielecki2019} for an interesting example where the ambiguity sets are recursive confidence regions for an unknown parameter of the disturbance distribution.

Let us now consider specifically deterministic Markov policies $\pi \in \Pi^M$ of the controller.  The subspace 
\[ \BB = \{ v \in \BB_b : \ v \text{ lower semicontinuous}  \}. \]
of $(\BB_b, \|\cdot\|_b )$ turns out to be the set of possible value functions under such policies. $(\BB, \|\cdot\|_b )$ is a complete metric space since the subset of lower semicontinuous functions is closed in $(\BB_b, \|\cdot\|_b )$. We define the following operators on $\BB_b$ and especially on $\BB$.

\begin{definition}\label{def:abstract_robust_operators}
	For $v \in \BB_b$ and Markov decision rules $d:E \to A$, $\gamma:D_n \to \Qc_{n+1}$ we define
	\begin{alignat*}{2}
		L_n v (x,a,\Q) &= \int c_n(x,a,T_n(x,a,Z_{n+1})) + v (T_n(x,a,Z_{n+1})) \dif \Q,\quad && (x,a,\Q) \in D_n\times \Qc_{n+1},\\
		\T_{n d \gamma} v(x) &= L_n v\big(x,d(x),\gamma(x,d(x))\big), && x \in E,\\
		\T_{n d} v(x) &= \sup_{\Q \in \Qc_{n+1}} L_n v(x,d(x),\Q), && x \in E,\\
		\T_n  v(x) &= \inf_{a \in D_n(x)} \sup_{\Q \in \Qc_{n+1}}  L_n v (x,a,\Q),&&x \in E.
	\end{alignat*}
\end{definition}

Note that the operators are monotone in $v$. Under Markov policies $\pi=(d_0, \dots, d_{N-1}) \in \Pi^M$ of the controller and $\gamma=(\gamma_0,\dots,\gamma_{N-1}) \in \Gamma^M$ of nature, the value iteration can be expressed with the help of these operators. In order to distinguish from the history-dependent case, we denote the value function here with $J$. Setting $J_{N\pi\gamma}(x)=c_N(x), \ x \in E,$ we obtain for $n=0, \dots, N-1$ and $x \in E$ with Proposition \ref{thm:abstract_robust_value_iteration}
\begin{align*}
J_{n\pi\gamma}(x) &= \int c_n\big(x,d_n(x),T_n(x,d_n(x),z_{n+1})\big) \\
& \phantom{= \int}\ + J_{n+1\pi\gamma} \big(T_n(x,d_n(x),z_{n+1})\big) \, \gamma_n(\dif z_{n+1}|x,d_n(x)) = \T_{nd_n\gamma_n} J_{n+1\pi\gamma}(x).
\end{align*} 
We define the robust value of Markov policy $\pi \in \Pi^M$ of the controller as
\begin{align*}
	J_{n\pi}(x) = \sup_{\gamma \in \Gamma^M} J_{n \pi \gamma}(x), \qquad x \in E.
\end{align*}

When calculating the robust value of a Markov policy of the controller it is no restriction to take the supremum only over Markov policies of nature.

\begin{corollary}\label{thm:abstract_robust_nature_markov}
	Let $\pi \in \Pi^M$. It holds for $n=0,\dots,N$ that $J_{n\pi}(x_n) = V_{n \pi}(h_n), \ h_n \in \H_n.$ I.e., we have the robust value iteration
	\begin{align*}
	J_{n\pi}(x) &= \sup_{\Q \in \Qc_{n+1}} \int c_n\big(x,d_n(x),T_n(x,d_n(x),Z_{n+1})\big)  + J_{n+1\pi} \big(T_n(x,d_n(x),Z_{n+1})\big) \dif \Q \\
	&= \T_{nd_n} J_{n+1\pi}(x).
	\end{align*}
	Moreover, there exists a Markovian $\epsilon$-optimal policy of nature and if the ambiguity sets $\Qc_{n+1}$ are all weak* closed even a Markovian optimal policy. 
\end{corollary}

\begin{proof}
	For $n=N$ the assertion is trivial. Assuming it holds at time $n+1$, it follows at time $n$ from Theorem \ref{thm:abstract_robust_nature_bellman} that
	\begin{align*}
		V_{n\pi}(h_n)&=\sup_{\Q \in \Qc_{n+1}} \int c_n\big(x,d_n(x),T_n(x,d_n(x),Z_{n+1})\big)  + J_{n+1\pi} \big(T_n(x,d_n(x),Z_{n+1})\big) \dif \Q\\
		&=J_{n\pi}(x_n).
	\end{align*}
	Replacing $\H_n\times A$ by $D_n$, the existence of ($\epsilon$-) optimal policies is guaranteed by the same arguments as in the proof of Theorem \ref{thm:abstract_robust_nature_bellman}. 
\end{proof}

Let us further define for $n=0, \dots, N$ the \emph{Markovian value function}
\begin{align*}
J_{n}(x) = \inf_{\pi \in \Pi^M} \sup_{\gamma \in \Gamma^M} J_{n\pi\gamma}(x) , \qquad x \in E.
\end{align*}
The next result shows that $J_n$ satisfies a Bellman equation and proves that an optimal policy of the controller exists and is Markov.

\begin{theorem}\label{thm:abstract_robust_finite}
	Let Assumptions \ref{ass:continuity_compactness}, \ref{ass:abstract_robust_finite} be satisfied.
	\begin{itemize}
		\item[a)] For $n=0, \dots, N-1$, it suffices to consider deterministic Markov policies both for the controller and nature, i.e. $V_n(h_n)=J_n(x_n)$ for all $h_n \in \mathcal{H}_n$. Moreover, $J_n\in \BB$ and satisfies the Bellman equation
		\begin{align*}
		J_N(x) &= c_N(x),\\
		J_n(x) &= \T_n J_{n+1}(x), \qquad x \in E.
		\end{align*}
		For $n= 0, \dots, N-1$ there exist Markov minimizers $d_n^*$ of $J_{n+1}$ and every sequence of such minimizers constitutes an optimal policy $\pi^*=(d_0^*,\dots,d_{N-1}^*)$ of the controller.
		\item[b)] If the  ambiguity sets $\Qc_n$ are weak* closed, there exist  maximizing decision rules $\gamma_n^*$ of nature, i.e. $J_{n}= \T_{d_n^* \gamma_n^*} J_{n+1}$ and every sequence of maximizers induces an optimal policy of nature $\gamma^*=(\gamma_0^*,\dots,\gamma_{N-1}^*) \in \Gamma^M$ satisfying $J_n = J_{n \pi^* \gamma^*}$.
	\end{itemize}
\end{theorem}

\begin{proof}
We proceed by backward induction. At time $N$ we have $V_N=J_N=c_N\in \BB$ due to semicontinuity and Assumption \ref{ass:abstract_robust_finite} (i). Now assume the assertion holds at time $n+1$. Using the robust value iteration (Corollary \ref{thm:abstract_robust_nature_markov}) we obtain at time $n$:
\begin{eqnarray*}
V_n(h_n) &=& \inf_{\pi \in \Pi^R} \sup_{\gamma \in \Gamma}  V_{n\pi\gamma}(h_n) = \inf_{\pi \in \Pi^R} V_{n\pi}(h_n)=\inf_{\pi \in \Pi^R} \T_{n\pi_n}  V_{n+1\pi}(h_n)\\
&\ge&  \inf_{\pi \in \Pi^R} \T_{n\pi_n}  V_{n+1}(h_n)= \inf_{\pi \in \Pi^R} \int \sup_{\Q \in \Qc_{n+1}} L_n J_{n+1}(x_n,a_n,\Q) \pi_n(\dif a_n|h_n)\\
&\ge & \inf_{a_n \in D_n(x_n)} \sup_{\Q \in \Qc_{n+1}}   L_n J_{n+1}(x_n,a_n,\Q) =\T_n J_{n+1}(x_n).
\end{eqnarray*}

	Here, objective and constraint depend on the history of the process only through $x_n$. Thus, given the existence of a minimizing Markov decision rule $d_n^*$ the last expression is equal to $\T_{n d_n^*} J_{n+1}(x_n)$. Again by the induction hypothesis, there exists an optimal Markov policy $\pi=(d_{n+1}^*,\dots,d_{N-1}^*) \in \Pi^M$ such that
\begin{eqnarray} \label{eq:thm3.10}
V_n(h_n)		\ge \T_n J_{n+1}(x_n)= \T_{n d_n^*} J_{n+1 \pi^*}(x_n)= J_{n\pi^*}(x_n)\geq J_n(x_n)\geq V_n(h_n).
\end{eqnarray}
		It remains to show the existence of a minimizing Markov decision rule $d_n^*$ and that $J_n \in \BB$. We want to apply Proposition 2.4.3 in \cite{BaeuerleRieder2011}. The set-valued mapping $E \ni x\mapsto D_n(x)$ is compact-valued and upper semicontinuous. Next, we show that $D_n \ni (x,a) \mapsto \sup_{\Q \in \Qc_{n+1}}   L_n v (x,a,\Q)$ is lower semicontinuous for every $v \in \BB$. Let $\{(x_k,a_k)\}_{k \in \N}$ be a convergent sequence in $D_n$ with limit $(x^*,a^*) \in D_n$. The function 
		\begin{eqnarray}\label{eq:thm3.10c} (x,a) \mapsto c_n\big(x,a,T_n(x,a,z)\big)+ v\big(T_n(x,a,z)\big)
		\end{eqnarray}
		 is lower semicontinuous for every $z\in \Z$. 
		The sequence of random variables $\{C_k\}_{n \in \N}$ given by
		$$C_k:= c_n\big(x_k,a_k,T_n(x_k,a_k,Z_{n+1})\big) +  v\big(T_n(x_k,a_k,Z_{n+1})\big)$$
		is bounded by some $\bar C \in L^p(\Omega_{n+1},\A_{n+1},\P_{n+1})$ by Lemma \ref{thm:abstract_robust_Lpbound}. Now, Fatou's Lemma yields for every $\Q \in \Qc_{n+1}$
		\begin{align*}
		\liminf_{k \to \infty} L_nv(x_k,a_k,\Q)
		&=	\liminf_{k \to \infty} \E^{\Q}\Big[ c_n\big(x_k,a_k,T_n(x_k,a_k,Z_{n+1})\big)+		
		 v\big(T_n(x_k,a_k,Z_{n+1})\big) \Big]\\
		&\geq \E^{\Q}\Big[ \liminf_{k \to \infty} c_n\big(x_k,a_k,T_n(x_k,a_k,Z_{n+1})\big)
+  v\big(T_n(x_k,a_k,Z_{n+1})\big) \Big]\\
		&\geq  \E^{\Q}\left[ c_n\big(x^*,a^*,T_n(x^*,a^*,Z_{n+1})\big) +  v\big(T_n(x^*,a^*,Z_{n+1})\big) \right]= L_nv(x^*,a^*,\Q),
		\end{align*}
		where the last inequality is by the lower semicontinuity of \eqref{eq:thm3.10c}. Thus, the function
		$ D_n \ni (x,a) \mapsto L_nv(x,a,\Q)$
		is lower semicontinuous for every $\Q \in \Qc_{n+1}$ and consequently $D_n \ni (x,a) \mapsto  \sup_{\Q \in \Qc_{n+1}}  L_n v (x,a,\Q)$ is lower semicontinuous as a supremum of lower semicontinuous functions. Now, Proposition 2.4.3 in \cite{BaeuerleRieder2011} yields the existence of a minimizing Markov decision rule $d_n^*$ at \eqref{eq:thm3.10} and that $J_n=\T_n J_{n+1}$ is lower semicontinuous. Furthermore, $J_n$ is bounded by $\lambda  b$ for some $\lambda \in \R_+$ due to Lemma \ref{thm:abstract_robust_bounds}. Thus $J_n \in \BB$.
Part b) follows from Theorem \ref{thm:abstract_robust_nature_bellman} b). \qedhere
\end{proof}

\section{Real Line as State Space}\label{sec:abstract_robust_real}

The  model has been introduced in Section \ref{sec:abstract_cost_model} with a general Borel space as state space. In order to solve the distributionally robust cost minimization problem in Section \ref{sec:abstract_robust_finite} we needed a continuous transition function despite having a semicontinuous model, cf.\ the proof of Theorem \ref{thm:abstract_robust_finite}. This assumption on the transition function can be relaxed to semicontinuity if the state space is the real line and the transition and one-stage cost function have some form of monotonicity. In some applications  this relaxation of the continuity assumption is relevant. Furthermore, a real state space can be exploited to address the  distributionally robust cost minimization problem with more specific techniques.
 In addition to Assumption \ref{ass:abstract_robust_finite}, we make the following assumptions in this section.

\begin{assumptions}\phantomsection \label{ass:abstract_robust_real}
	\begin{itemize}
		\item[(i)] The state space is the real line $E=\R$.
		\item[(ii)] The model data satisfies the Continuity and Compactness Assumptions \ref{ass:continuity_compactness} with the transition function $T_n$ being lower semicontinuous instead of continuous.
		\item[(iii)] The model data has the following monotonicity properties:
		\begin{itemize}
			\item[(iii a)] The set-valued mapping $\R \ni x \mapsto D_n(x)$ is decreasing.
			\item[(iii b)] The function $\R \ni x \mapsto T_n(x,a,z)$ is increasing for all $a\in D_n(x), z\in \Z$.
			\item[(iii c)] The function $\R^2  \ni (x,x')\mapsto c_n(x,a,x')$ is increasing for all $a\in D_n(x)$.
			\item[(iii d)] The function $\R \ni x\mapsto c_N(x)$ is increasing.
		\end{itemize}
	\end{itemize}
\end{assumptions}

With the real line as state space a simple separation condition is sufficient for Assumption \ref{ass:abstract_robust_finite} (ii).

\begin{corollary}\label{thm:abstract_robust_separation}
	Let there be upper semicontinuous functions $\vartheta_{n,1},\vartheta_{n,2}:D \to \R_+$ and measurable functions $\Theta_{n,1},\Theta_{n,2}: \Z \to \R_+$ which fulfil $\Theta_{n,1}(Z_{n+1}),\Theta_{n,2}(Z_{n+1}) \in L^p(\Omega_n,\A_n,\P_n)$ and
	\begin{align*}
	|c_n(x,a,T_n(x,a,z))| \leq \vartheta_{n,1}(x,a) + \Theta_{n,1}(z),  \qquad  b(T_n(x,a,z)) \leq \vartheta_{n,2}(x,a) + \Theta_{n,2}(z)
	\end{align*}
	for every $(x,a,z) \in D \times \Z$. Then Assumption \ref{ass:abstract_robust_finite} (ii) is satisfied.
\end{corollary}
\begin{proof}
	Let $(\bar x, \bar a) \in D_n$. We can choose $\epsilon >0$ arbitrarily. The set $S=[\bar x- \epsilon, \bar x+\epsilon] \times D_n(\bar x - \epsilon)$ is compact w.r.t.\ the product topology. Moreover, $B_\epsilon(\bar x,\bar a) \cap D_n \subseteq S$ since the set-valued mapping $D_n(\cdot)$ is decreasing. Due to upper semicontinuity there exist $(x_i, a_i) \in S$ such that $\vartheta_{n,i}( x_i, a_i) = \sup_{(x,a) \in S} \vartheta_{n,i}(x,a)$, $i=1,2$. Hence, one can define
	\[ \Theta_{n,i}^{\bar x, \bar a}(\cdot) := \vartheta_{n,i}(x_i, a_i) + \Theta_{n,i}(\cdot), \quad i=1,2 \]
	and Assumption \ref{ass:abstract_robust_finite} (ii) is satisfied.
\end{proof}

The question is, how replacing Assumption \ref{ass:continuity_compactness} (ii) by Assumption \ref{ass:abstract_robust_real}  affects the validity of all previous results. The only result that was proven using the continuity of the transition function $T_n$ in $(x,a)$ and not only its measurability is Theorem \ref{thm:abstract_robust_finite}. All other statements are unaffected.

\begin{proposition}\label{thm:abstract_robust_real}
	The assertions of Theorem \ref{thm:abstract_robust_finite} still hold when we replace Assumption \ref{ass:continuity_compactness} by  Assumption \ref{ass:abstract_robust_real}. Moreover, $J_n$ and $J$ are increasing. The set of potential value functions can therefore be replaced by 
	\[ \BB = \{ v \in \BB_b : \ v \text{ lower semicontinuous and increasing}  \}. \]
\end{proposition}
\begin{proof}
In the  proof of Theorem \ref{thm:abstract_robust_finite} the continuity of $T_n$ is used to show that $D_n \ni (x,a) \mapsto L_n v(x,a)$ is lower semicontinuous for every $v \in \BB$. Due to the monotonicity assumptions $$D_n \ni (x,a) \mapsto c_n\big(x,a,T_n(x,a,z_{n+1})\big)+  v\big(T_n(x,a,z_{n+1})\big)$$ is lower semicontinuous for every $z_{n+1}\in \Z$. Now the lower semicontinuity of $D_n \ni (x,a) \mapsto L_nv(x,a)$ and the existence of a minimizing decision rule follows as in the proof of Theorem \ref{thm:abstract_robust_finite}. The fact that $\T_n v$ is increasing for every $v \in \BB$ follows as in Theorem 2.4.14 in  \cite{BaeuerleRieder2011}.
\end{proof}

In the following Section \ref{sec:abstract_robust_minimax} we use a minimax approach as an alternative way to solve the Bellman equation of the distributionally robust cost minimization problem and to study its game theoretical properties. Subsequently in Section \ref{sec:special_sets}, we also consider special choices of the ambiguity set which are advantageous for solving the optimization problem.

\section{Minimax Approach and Game Theory}\label{sec:abstract_robust_minimax}
We assume $E = \R$ as in the last section.
Compared to a risk-neutral Markov Decision Model, the Bellman equation of the robust model (see Theorem \ref{thm:abstract_robust_finite})
has the additional complication that a supremum over possibly uncountably many expectations needs to be calculated. This can be a quite challenging task. Therefore, it may be advantageous to interchange the infimum and supremum. For instance, in applications it may be possible to infer structural properties of the optimal actions independently from the probability measure $\Q$ after the interchange. Based on the minimax theorem by \cite{Sion1958}, cf.\ Appendix Theorem \ref{thm:A:minimax}, this section presents a criterion under which the interchange of infimum and supremum is possible.

\begin{lemma}\label{thm:abstract_robust_convex}
Let Assumptions \ref{ass:abstract_robust_finite}, \ref{ass:abstract_robust_real} be satisfied.	Let $A$ be a subset of a vector space, $D_n$ be a convex set, $x\mapsto c_N(x)$, $(x,a) \mapsto T_n(x,a,z)$ be convex as well as $ (x,a,x') \mapsto c_n(x,a,x')$. Then the value functions $J_n$ and the limit value function $J$ are convex.   
\end{lemma}

\begin{proof}
	The proof is by backward induction. $J_N=c_N$ is convex by assumption. Now assume that $J_{n+1}$ is convex. Recall that $J_{n+1}$ is increasing (Proposition \ref{thm:abstract_robust_real}). Hence, for every $z\in\Z$ the function
	\begin{align*}
	D_n \ni (x,a) \mapsto c_n\big(x,a,T_n(x,a,z)\big) +  J_{n+1}\big(T_n(x,a,z)\big)
	\end{align*}
	is convex as  a composition of an increasing convex with a convex function. By the linearity of expectation, 
	\begin{align} \label{eq:abstract_robust_convex}
	D_n \ni (x,a) \mapsto \E^\Q\Big[ c_n\big(x,a,T_n(x,a,Z_{n+1})\big) + J_{n+1}\big(T_n(x,a,Z_{n+1})\big) \Big] 
	\end{align}
	is convex for every $\Q \in \Qc_{n+1}$. As the pointwise supremum of a collection of convex functions is convex, we obtain convexity of $D_n \ni (x,a) \mapsto \sup_{\Q \in \Qc_{n+1}} LJ_{n+1}(x,a,\Q)$. Now, Proposition 2.4.18 in  \cite{BaeuerleRieder2011} yields the assertion. 
\end{proof}

The assumptions of Lemma \ref{thm:abstract_robust_convex} are subsequently referred to as \emph{convex model}.

\begin{theorem}\label{thm:abstract_robust_minimax_convex}
Let Assumptions \ref{ass:abstract_robust_finite}, \ref{ass:abstract_robust_real} be satisfied.		In a convex model we have for all $n=0,\dots,N-1$
	\begin{align*}
	J_n(x) &= \inf_{a \in D_n(x)} \sup_{\Q \in \Qc_{n+1}} L_n J_{n+1}(x,a,\Q) =  \sup_{\Q \in \Qc_{n+1}} \inf_{a \in D_n(x)} L_n J_{n+1}(x,a,\Q), \qquad x \in \R.
	\end{align*} 
\end{theorem}

\begin{proof}
	Let $x \in \R$ be fixed and define $f: D_n(x) \times \Qc_{n+1} \to \R$,
	\[ f(a,\Q) = L_n v(x,a,\Q) =  \E^\Q\Big[ c_n\big(x,a,T_n(x,a,Z_{n+1})\big) +  J_{n+1}\big(T_n(x,a,Z_{n+1})\big) \Big]. \]
	The function $f$ is convex in $a$ by \eqref{eq:abstract_robust_convex} and linear in $\Q$, i.e. especially concave. Furthermore, the set $D_n(x)$ is compact and it has been shown in the proof of Theorem \ref{thm:abstract_robust_finite} that $f$ is lower semicontinuous in $a$. Hence, the assertion follows from Theorem \ref{thm:A:minimax} a).
\end{proof}

\begin{remark}
	The interchange of infimum and supremum in Theorem \ref{thm:abstract_robust_minimax_convex} is based on Sion's Minimax Theorem \ref{thm:A:minimax}, which requires convexity of the function
	\begin{align}\label{eq:abstract_robust_minimax_randomize_counterexample_0}
		a \mapsto \int c_n\big(x,a,T_n(x,a,z)\big) +  J_{n+1}\big(T(x,a,z)\big) \, \Q(\dif z)
	\end{align}
	for every $(x,\Q) \in \R \times \Qc$. This can be guaranteed by a convex model (cf.\ Lemma \ref{thm:abstract_robust_convex}) which means that several components of the decision model need to have some convexity property. However, these assumptions are quite restrictive. Resorting to randomized actions is a standard approach to convexify (or more precisely linearize) the function \eqref{eq:abstract_robust_minimax_randomize_counterexample_0} without assumptions on the model components. Let $\Pc(D_n(x))$ be the set of all probability measures on $D_n(x)$. Then it follows from Sion's Theorem \ref{thm:A:minimax} that\begin{align}\label{eq:abstract_robust_minimax_randomize_counterexample_1}
&\inf_{\mu \in \Pc(D_n(x))} \sup_{\Q \in \Qc_{n+1}} 		\int L_n J_{n+1}(x,a,\Q) \mu(da)
=		  \sup_{\Q \in \Qc_{n+1}} \inf_{\mu \in \Pc(D_n(x))} \int L_n J_{n+1}(x,a,\Q) \mu(da) \\
		&=  \sup_{\Q \in \Qc_{n+1}} \inf_{a \in D_n(x)}  L_n J_{n+1}(x,a,\Q). \notag
	\end{align}
	The last equality holds since $a \mapsto c_n\big(x,a,T_n(x,a,z)\big) +  J_{n+1}\big(T_n(x,a,z)\big)$ is lower semicontinuous (cf.\ the proof of Theorem \ref{thm:abstract_robust_finite}) and $D_n(x)$ is compact. This appears to be a very elegant solution for the interchange problem, but unfortunately, the Bellman equation of the distributionally robust cost minimization problem \eqref{eq:opt_crit_abstract_robust_finite} under a randomized action of the controller is given by
	\begin{align}\label{eq:abstract_robust_minimax_randomize_counterexample_2}
		J_n(x) &= \inf_{\mu \in \Pc(D_n(x))} \int \sup_{\Q \in \Qc_{n+1}} \int c_n\big(x,a,T_n(x,a,z)\big) +  J_{n+1}\big(T_n(x,a,z)\big) \, \Q(\dif z) \, \mu(\dif a)\\
		&= \inf_{a \in D_n(x)} \sup_{\Q \in \Qc_{n+1}}  L_n J_{n+1}(x,a,\Q), \notag
	\end{align}
	cf.\ Theorems \ref{thm:abstract_robust_nature_bellman} and \ref{thm:abstract_robust_finite}, and \eqref{eq:abstract_robust_minimax_randomize_counterexample_1} does in general not equal \eqref{eq:abstract_robust_minimax_randomize_counterexample_2}. Recall that in our model nature is allowed to react to any realization of the controller's action. This was crucial to obtain a robust value iteration in Theorem \ref{thm:abstract_robust_nature_bellman}. In contrast to that, \eqref{eq:abstract_robust_minimax_randomize_counterexample_1} means that nature maximizes only knowing the distribution of the controller's action. However, in general \eqref{eq:abstract_robust_minimax_randomize_counterexample_1} $\neq$ \eqref{eq:abstract_robust_minimax_randomize_counterexample_2} as will be shown in the next example.
\end{remark}

\begin{example}	
	 In order to  see that \eqref{eq:abstract_robust_minimax_randomize_counterexample_1} $\neq$ \eqref{eq:abstract_robust_minimax_randomize_counterexample_2} and that infimum and supremum cannot be interchanged in general, consider the simple static counter example $N=1$, $E=\R$, $A=[0,1]$, $D=\R \times A$, $Z \sim \text{Bin}(1,p)$, $p\in [0,1] = \Qc$, $T(x,a,z)=-(a-z)^2$ and $c(x,a,x')=x'$. It is readily checked that Assumption \ref{ass:abstract_robust_real} is satisfied. Especially, one has constant bounding functions. In this example \eqref{eq:abstract_robust_minimax_randomize_counterexample_2} equals
	\begin{align*}
	&	\inf_{a \in [0,1]} \sup_{p \in [0,1]} \E^p\big[c(x,a,T(x,a,Z))\big]= \inf_{a \in [0,1]} \sup_{p \in [0,1]} -(1-p)a^2 -p(a-1)^2\\
		&= - \sup_{a \in [0,1]} \inf_{p \in [0,1]} (1-p)a^2 +p(a-1)^2= - \sup_{a \in [0,1]} \min\{a^2,(1-a)^2\} =-\frac14.
	\end{align*}
	If the controller chooses $\mu \sim \U(0,1)$, then \eqref{eq:abstract_robust_minimax_randomize_counterexample_1} must be less or equal than
	\begin{align*}
		\sup_{p \in [0,1]} \int_0^1 -(1-p)a^2 -p(a-1)^2 \, \dif a &= \sup_{p \in [0,1]} -\frac13 (1-p) - \frac13 p = -\frac13.
	\end{align*}
	Indeed we obtain here $	\sup_{p \in [0,1]}\inf_{a \in [0,1]}  \E^p\big[c(x,a,T(x,a,Z))\big]=-\frac12$.
	The approach to interchange infimum and supremum through a linearization with randomized actions works when nature only observes the distribution and not the realization of the controller's action. Also note that the situation here is quite different to classical two-person zero-sum games. The fact that infimum and supremum cannot be interchanged is a consequence of the asymmetric mover/follower situation. The example above with $P_{xa}^p=(1-p)a^2 +p(a-1)^2$ and $[0,1]$ discretized can also be used as  a counter-example within the setting of \cite{NilimGhaoui2005} and \cite{wiesemann2013robust}.
\end{example}

As mentioned before, the distributionally robust cost minimization model can be interpreted as a dynamic game with nature as the controller's opponent. Since nature chooses her action after the controller, observing his action but not being restricted by it, there is a (weak) \emph{second-mover advantage} by construction of the game. The fact that infimum and supremum in the Bellman equation can be interchanged means that the second-mover advantage vanishes in the special case of a convex model.

Let additionally the one-stage ambiguity sets $\Qc_n$ be weak* closed.  Now, the conditions of Theorem \ref{thm:A:minimax} b) are fulfilled, too. Then, the ambiguity set is weak* compact by Lemma \ref{thm:ambiguity_sep_metrizable} and by Lemma \ref{thm:abstract_robust_Lpbound} we have that $c_n\big(x,a,T_n(x,a,Z_{n+1})\big) +   J_{n+1}\big(T_n(x,a,Z_{n+1})\big)$ is in $L^p$. Thus, $\Q \mapsto L_n J_{n+1}(x,a,\Q)$ is weak* continuous for every $(x,a) \in D$.  This yields that $(a,\Q) \mapsto L_n J_{n+1}(x,a,\Q)$ satisfies the minimax-equality
\[ \min_{a \in D_n(x)} \max_{\Q \in \Qc_{n+1}} L_n J_{n+1}(x,a,\Q) =  \max_{\Q \in \Qc_{n+1}} \min_{a \in D_n(x)} L_n J_{n+1}(x,a,\Q) \]
and Lemma 2.105 in \cite{BarbuPrecupanu2012} implies  that for every $x \in \R$ the function has a saddle point $(a^*,\Q^*)$, i.e.
\[ L_nJ_{n+1}(x,a^*,\Q)  \leq  L_n J_{n+1}(x,a^*,\Q^*) \leq L_n J_{n+1}(x,a,\Q^*)  \]
for all $a \in D(x)$ and $\Q \in \Qc_{n+1}$. Such a saddle point constitutes a \emph{Nash equilibrium} in the subgame scenario $X_n=x$. We will show that Nash equilibria exist not only in one-stage subgames but also globally.

Before, let us introduce a modification of the game against nature where nature instead of the controller moves first. Given a policy of nature, the controller faces an arbitrary but fixed probability measure in each scenario $X_n=x$. Thus, the inner optimization problem is a risk-neutral MDP and it follows from standard theory  that it suffices for the controller to consider deterministic Markov policies.  Clearly, the controller's optimal policy will depend on the policy that nature has chosen before. It will turn out to be a pointwise dependence on the actions of nature. To clarify this and for comparability with the original game \eqref{eq:opt_crit_abstract_robust_finite}, where the controller moves first, we distinguish the following types of Markov strategies of the controller
\begin{align*}
\Pi(\R)  = & \Pi^M =  \ \left\{ \pi=(d_0,\dots,d_{N-1}) | \, d_n: \R \to A \text{ measurable}, \, d_n(x) \in D_n(x),\, x \in \R \right\}\\
\Pi(\R,\Qc) = & \ \left\{ \pi=(d_0,\dots, d_{N-1}) | \, d_n: \R \times \Qc_{n+1} \to A \text{ measurable}, \, d_n(x,\Q) \in D_n(x),\, x \in \R \right\}
\end{align*}
and of nature
\begin{align*}
\Gamma(\R) = & \ \left\{ \gamma=(\gamma_0,\dots,\gamma_{N-1}) | \ \gamma_n: \R \to \Qc_{n+1}  \text{ measurable} \right\}\\
\Gamma(\R,A) = \Gamma^M = & \ \left\{ \gamma=(\gamma_0,\dots,\gamma_{N-1}) | \ \gamma_n: \R \times A \to \Qc_{n+1}  \text{ measurable} \right\}.
\end{align*}
The value $J_{n\pi\gamma}$ of a pair of Markov policies $(\gamma,\pi) \in \Gamma(\R)\times \Pi(\R,\Qc)$ is defined as in \eqref{eq:abstract_robust_policy_value}. The bounds in Lemma \ref{thm:abstract_robust_bounds}   apply since the proofs do not use properties of the policies. The game under consideration is 
\begin{align} \label{eq:abstract_robust_nature_first_finite}
\wt J_n(x)= \sup_{\gamma \in \Gamma(\R)} \inf_{\pi \in \Pi(\R,\Qc)} J_{n\pi\gamma}(x), \qquad x \in \R, \ n=0,\ldots, N .
\end{align}
For clarity, we mark all quantities of the game where nature moves first which differ from the respective quantity of the original game with a tilde. The \emph{value of a policy of nature} $\gamma \in \Gamma(\R)$ is defined as
\begin{align*}
\wt J_{n\gamma}(x)= \inf_{\pi \in \Pi(\R,\Qc)} J_{n\pi\gamma}(x), \qquad x \in \R.
\end{align*}
The Bellman operator on $\BB$ can be introduced in the usual way:
\begin{align*}
\wt \T_n v(x) &= \sup_{\Q \in \Qc_{n+1}} \inf_{a \in D_n(x)} L_n v (x,a,\Q), \quad x \in \R. 
\end{align*}

\begin{theorem}\label{thm:abstract_robust_nature_first}
	Let Assumptions \ref{ass:abstract_robust_finite}, \ref{ass:abstract_robust_real} be satisfied, the ambiguity sets  $\Qc_{n+1}$ be weak* closed and the model be convex.
	\begin{itemize}
		\item[a)]  $\wt J_n\in \BB$ for $ n=0,\ldots, N$ and they satisfy the Bellman equation
		\begin{align*}
		\wt J_{N}(x) &= c_N(x),\\
		\wt J_{n}(x) &= \wt \T_n \wt J_{n+1}(x), \qquad x \in \R.
		\end{align*}
		There exist decision rules $\tilde \gamma_n:\R \to \Qc_{n+1}$ of nature and $\tilde d_n:\R\times \Qc_{n+1} \to A$ of the controller such that $\wt J_n (x) = \T_{\tilde d_n \tilde \gamma_n} \wt J_{n+1}(x)$ and all sequences of such decision rules induce an optimal policy pair $\tilde \gamma=(\tilde \gamma_0,\dots,\tilde \gamma_{n-1}) \in \Gamma(\R)$ and $\tilde \pi=(\tilde d_0,\dots, \tilde d_{n-1}) \in \Pi(\R,\Qc)$ i.e. $\widetilde J_n = J_{n \tilde \pi \tilde \gamma}$. 
		\item[b)]  We have that $J_n=\wt J_n = \wt J_{n\tilde \gamma}$.
	\end{itemize}
\end{theorem}

\begin{proof}
 We have for $n$ and $x \in \R$:
		\begin{align}\label{eq:abstract_robust_nature_first_proof1}
		J_n(x) &= \inf_{\pi \in \Pi(\R)} \sup_{\gamma \in \Gamma(\R,A)} J_{n\pi\gamma}(x)  \geq \inf_{\pi \in \Pi(\R)} \sup_{\gamma \in \Gamma(\R)} J_{n\pi\gamma}(x) \geq  \sup_{\gamma \in \Gamma(\R)} \inf_{\pi \in \Pi(\R)}  J_{n\pi\gamma}(x)\notag\\
		&\geq \sup_{\gamma \in \Gamma(\R)} \inf_{\pi \in \Pi(\R,\Q)}  J_{n\pi\gamma}(x) = \wt J_n(x).
		\end{align}
Let $\pi^* = ( d_0^*, \dots,  d_{N-1}^*) \in \Pi(\R)$ and $\gamma^* = (\gamma_0^*,\dots, \gamma_{N-1}^*) \in \Gamma(\R,A)$ be optimal strategies for the original game \eqref{eq:opt_crit_abstract_robust_finite}. The existence is guaranteed by Theorem \ref{thm:abstract_robust_finite}.
		Then $\tilde \gamma = (\tilde \gamma_0,\dots, \tilde \gamma_{N-1})$ defined by $\tilde \gamma_n:= \gamma_n^*(\cdot, d_n^*(\cdot))$
		lies in $\Gamma(\R)$ since the decision rules are well defined as compositions of measurable maps.\\
We prove all statements simultaneously by induction. In particular we show that there exists a policy $\tilde \pi \in \Pi(\R,\Qc)$ such that
$$ J_n = \wt J_{n\tilde \gamma}= \wt J_{n\tilde \pi \tilde \gamma}.$$ 
We show next that $J_n \le \wt J_n$. From Theorem \ref{thm:abstract_robust_minimax_convex} and the induction hypothesis we obtain
\begin{eqnarray*}
J_n(x) &=& L_n J_{n+1}(x,d_n^*(x),\tilde \gamma_n(x)) = \inf_{a\in D_n(x)} L_n J_{n+1}(x,a,\tilde \gamma_n(x)) =  \inf_{a\in D_n(x)} L_n \wt J_{n+1}(x,a,\tilde \gamma_n(x)).
\end{eqnarray*}
Observe that again by the induction hypothesis
\begin{eqnarray}\nonumber
\wt J_{n\tilde \gamma}(x) &=& \inf_{\pi\in \Pi(\R,\Qc)} J_{n\pi \tilde \gamma}(x) =  \inf_{\pi\in \Pi(\R,\Qc)} L_n J_{n+1 \pi\tilde \gamma} (x, d_n(x,\tilde \gamma_n(x)),\tilde \gamma_n(x))\\ \nonumber
&\ge &  \inf_{a\in D_n(x)}  L_n \wt J_{n+1\tilde \gamma} (x,a,\tilde \gamma_n(x)) =L_n \wt J_{n+1\tilde \gamma} (x,\tilde d_n(x,\tilde \gamma_n (x)),\tilde \gamma_n(x))\\ \label{eq:thm5.4}
&=& L_n \wt J_{n+1\tilde \pi \tilde \gamma} (x,\tilde d_n(x,\tilde \gamma_n (x)),\tilde \gamma_n(x))= \wt J_{n\tilde \pi \tilde \gamma}\ge \wt J_{n\tilde \gamma}.
\end{eqnarray}
We will show below the existence of a minimizer $\tilde d_n$ in the second line is justified.
Thus we get equality in the expression above. Combining the previous two equations above we finally get
\begin{eqnarray*}
J_n(x) &=&  \inf_{a\in D_n(x)} L_n \wt J_{n+1}(x,a,\tilde \gamma_n(x)) =  \wt J_{n\tilde \gamma}\le \wt J_n(x).
\end{eqnarray*}
In total we have shown that $J_n = \wt J_n = \tilde J_{n\tilde \gamma}=\tilde J_{n\tilde\pi\tilde\gamma}$.  The joint Bellman equation for the controller and nature $ \wt J_{n} = \wt \T \wt J_{n+1}$ follows from Theorem \ref{thm:abstract_robust_minimax_convex}.

Finally, we verify the existence of a minimizer $\tilde d_n$. 
 Let $\{(x_n,a_n,\Q_n)\}_{n \in \N}$ be a convergent sequence in  $\R \times A \times \Qc$ with limit $(x^*,a^*,\Q^*)$.
By dominated convergence (Lemma \ref{thm:abstract_robust_Lpbound}) and the lower semicontinuity of  $D_n \ni (x,a) \mapsto c_n\big(x,a,T_n(x,a,z_{n+1})\big)+ v\big(T_n(x,a,z_{n+1})\big)$ for any $v\in \BB$ we obtain that
the increasing sequence of random variables $\{C_m\}_{m}$ given by
		$$C_m:= \inf_{k \geq m} c_n\big(x_k,a_k,T_n(x_k,a_k,Z_{n+1})\big) +  v\big(T_n(x_k,a_k,Z_{n+1})\big)$$ satisfies
\begin{align*}
		C_m \overset{L^p}{\longrightarrow} C^* \geq c_n\big(x^*,a^*,T_n(x^*,a^*,Z_{n+1})\big) +  v\big(T(x^*,a^*,Z_{n+1}))\big).
		\end{align*}
		Since $\Qc_{n+1}$ is norm bounded, Corollary 6.40 in \cite{AliprantisBorder2006} yields that the duality
		\[ (X,\Q) \mapsto \E^\Q[X] \]
		restricted to $L^p(\Omega_n,\A_n,\P_n)\times \Qc_{n+1}$ is jointly continuous, where $L^p(\Omega_n,\A_n,\P_n)$ is considered with the norm topology and $\Qc_{n+1}$ with the weak* topology.		
 Thus, we get
		\begin{align*}
&		\liminf_{m \to \infty}  L_nv(x_m,a_m,\Q_m) 		 \geq \liminf_{m \to \infty} \E^{\Q_m}\Big[ \inf_{k \geq m} c_n\big(x_k,a_k,T_n(x_k,a_k,Z_{n+1})\big) +   v\big(T_n(x_k,a_k,Z_{n+1})\big) \Big]\\
		&= \lim_{m \to \infty} \E^{\Q_m}[ C_m ]= \E^{\Q^*}[ C^* ]	\geq \E^{\Q^*}\Big[ c_n\big(x^*,a^*,T_n(x^*,a^*,Z_{n+1})\big) +  v\big(T_n(x^*,a^*,Z_{n+1}))\big) \Big]\\
		& = L_nv(x^*,a^*,\Q^*),
		\end{align*}
		which establishes the joint lower semicontinuity of $L_nv(\cdot)$.
		Note that $\wt J_{n+1 \tilde \gamma} \in \BB$ and $x \mapsto D_n(x)$ is  compact-valued and upper semicontinuous. Hence, it follows from Theorem 2.4.3 in \cite{BaeuerleRieder2011}  that there exists a  minimizing decision rule $\tilde d_n:\R\times\Qc_{n+1} \to A$ at \eqref{eq:thm5.4} and that 
		\begin{align*}
		\R \times \Qc_{n+1} \ni (x,\Q) \mapsto  \inf_{a \in D(x)}  L_n\wt J_{n+1 {\tilde \gamma}}(x,a,\Q) =  L_n\wt J_{n+1 {\tilde \gamma}}(x,\tilde d_n(x,\Q),\Q)
		\end{align*}
		is lower semicontinuous. This completes the proof.\qedhere

\end{proof}

\begin{remark}
Note that in contrast to classical zero-sum games we obtain the existence of deterministic policies.
\end{remark}

As a direct consequence we get the existence of Nash equilibria on policy level. 

\begin{corollary}\label{thm:abstract_robust_Nash}
	Consider a convex model with weak* closed ambiguity sets $\Qc_{n+1}$ and Assumptions \ref{ass:abstract_robust_finite}, \ref{ass:abstract_robust_real} fulfilled.  For $x \in \R$ we get
		\begin{align*}
		J_n(x) = \min_{\pi \in \Pi(\R) }\max_{\gamma \in \Gamma(\R,A)} J_{n\pi\gamma}(x) = \max_{\gamma \in \Gamma(\R)} \min_{\pi \in \Pi(\R,\Qc)} J_{n\pi\gamma}(x) = \tilde J_n(x).
		\end{align*}
		Consequently, we even obtain
		\[ J_n(x) = \min_{\pi \in \Pi(\R) }\max_{\gamma \in \Gamma(\R)} J_{n\pi\gamma}(x) = \max_{\gamma \in \Gamma(\R)} \min_{\pi \in \Pi(\R)} J_{n\pi\gamma}(x). \]
\end{corollary}

\begin{proof}
 Theorem \ref{thm:abstract_robust_nature_first} implies equality in \eqref{eq:abstract_robust_nature_first_proof1}, i.e.
		\begin{align}\label{eq:abstract_robust_Nash_proof}
		J_n(x) &= \min_{\pi \in \Pi(\R)} \max_{\gamma \in \Gamma(\R,A)} J_{n\pi\gamma}(x) = \inf_{\pi \in \Pi(\R)} \sup_{\gamma \in \Gamma(\R)} J_{n\pi\gamma}(x) \notag\\
		&=  \sup_{\gamma \in \Gamma(\R)} \inf_{\pi \in \Pi(\R)}  J_{n\pi\gamma}(x)= \max_{\gamma \in \Gamma(\R)} \min_{\pi \in \Pi(\R,\Q)}  J_{n\pi\gamma}(x) = \wt J_n(x). 
		\end{align}
		It remains to find optimal policies for the second and fourth equation of \eqref{eq:abstract_robust_Nash_proof}. Let
		\begin{align*}
		\pi^* &= (d_0^*, \dots, d_{N-1}^*) \in \Pi(\R), &  \gamma^* &= (\gamma_0^*,\dots, \gamma_{N-1}^*) \in \Gamma(\R,A)\\
		\text{and} \qquad \qquad \tilde \gamma &= (\tilde \gamma_0,\dots, \tilde \gamma_{N-1}) \in \Gamma(\R) & \tilde \pi &= (\tilde d_0, \dots, \tilde d_{N-1}) \in \Pi(\R,\Qc)
		\end{align*}
		be optimal strategies for the first and fifth equation of \eqref{eq:abstract_robust_Nash_proof}, respectively, which exist by Proposition \ref{thm:abstract_robust_real} and Theorem \ref{thm:abstract_robust_nature_first}. Then
		\[ \inf_{\pi \in \Pi(\R)} \sup_{\gamma \in \Gamma(\R)} J_{n\pi\gamma} = \sup_{\gamma \in \Gamma(\R)} \inf_{\pi \in \Pi(\R)}  J_{n\pi\gamma}  \]
		is attained by the admissible strategy pair $(\hat \pi,\hat \gamma) \in \Pi(\R) \times \Gamma(\R)$ which can be defined by $\hat d_n := d_n^*$ and $\hat \gamma_n:= \gamma_n^*(\cdot,d_n^*(\cdot))$ or alternatively by $\hat d_n := \tilde d_n(\cdot,\tilde \gamma_n(\cdot))$ and $\hat \gamma_n:= \tilde \gamma_n$ for $ n=0,\dots, N-1$.\qedhere
\end{proof}

\section{Special Ambiguity Sets}\label{sec:special_sets}
In this section, we consider some special choices for the ambiguity set $\Qc_{n+1}$ which simplify solving the Markov Decision Problem \eqref{eq:opt_crit_abstract_robust_finite}  or allow for structural statements about the solution. We assume a real-valued state space here.\medskip

\textbf{Convex hull.} It does not change the optimal value of the optimization problems if a given ambiguity set $\Qc_{n+1}$ is replaced by its convex hull $\conv(\Qc_{n+1})$ or its closed convex hull $\cconv(\Qc_{n+1})$, where the closure is with respect to the weak* topology. Clearly, to demonstrate this, it suffices to compare the corresponding Bellman equations.
\begin{lemma}\label{thm:abstract_robust_ambiguity_convex_hull}
	Let $\Qc_{n+1}$ be any norm bounded ambiguity set. Then it holds for all $v \in \BB$ and $x \in \R$
	\begin{align*}
	\inf_{a \in D_n(x)} \sup_{\Q \in \Qc_{n+1}} L_n v(x,a,\Q)
	= \inf_{a \in D(x)} \sup_{\Q \in \conv(\Qc_{n+1})}L_nv(x,a,\Q)
	= \inf_{a \in D(x)} \sup_{\Q \in \cconv(\Qc_{n+1})} L_nv(x,a,\Q).
	\end{align*}
\end{lemma}
\begin{proof}
	Fix $(x,a) \in D_n$. The function $\Q \mapsto L_n v(x,a,\Q)$ is linear. Thus, for a generic element $\Q = \sum_{i=1}^n \lambda_i \Q_i \in \conv(\Qc_{n+1})$ we have
	\[ L_nv\left(x,a, \sum_{i=1}^m \lambda_i \Q_i \right) = \sum_{i=1}^m \lambda_i L_nv(x,a,\Q_i) \leq \max_{i=1,\dots,m} L_n v(x,a,\Q_i), \]
	i.e. there can be no improvement of the supremum on the convex hull.  We also have  that $\cconv(\Qc_{n+1})$ is metrizable and therefore coincides with the limit points of sequences in $\conv(\Qc_{n+1})$. Since $\Q \mapsto L_nv(x,a,\Q)$ is weak* continuous (cf.\ proof of Theorem \ref{thm:abstract_robust_nature_bellman}), the supremum cannot be improved on the closure either. 
\end{proof}

\textbf{Integral stochastic orders on $\Qc_{n+1}$.} Following an idea of \cite{Mueller1997}, one can define integral stochastic orders on the ambiguity $\Qc_{n+1}$ set by
\begin{align*}
\Q_1 \leq_{\BB,x,a} \Q_2 \  : \Longleftrightarrow \ & L_n v(x,a,\Q_1) \leq  L_n v(x,a,\Q_2) \quad \text{for all } v \in \BB
\intertext{ where $(x,a) \in D_n$ is fixed and}
\Q_1  \leq_{\BB} \Q_2 \  : \Longleftrightarrow \ & \Q_1 \leq_{\BB,x,a} \Q_2 \quad \text{for all } (x,a) \in D_n.
\end{align*}
If there exists a maximal element with respect to one of these stochastic orders, this probability measure is an optimal action for nature (in the respective scenario). The proof of the next lemma follows directly from the Bellman equation and the definition of the orderings.

\begin{lemma} \phantomsection \label{thm: abstract_robust_integral_stoch_order}
	\begin{enumerate}
		\item If there exists a maximal element $\Q_{n,x,a} \in \Qc_{n+1}$  w.r.t.\ $\leq_{\BB,x,a}$ for every $(x,a) \in D_n$, then $\gamma=(\gamma_0,\dots,\gamma_{N-1})$ with $\gamma_n(x,a):= \Q_{n,x,a}$ defines an optimal policy.
		\item If there exists a maximal element $\Q^*_n \in \Qc_n$  w.r.t.\ $\leq_{\BB}$, then $\gamma=(\gamma_0,\dots,\gamma_{N-1})$ with $\gamma_n \equiv \Q^*_n$ defines a constant optimal action of nature. 
		That is, \eqref{eq:opt_crit_abstract_robust_finite}  can be reformulated to risk-neutral MDPs under the probability measure $\Q^*_n$. 
	\end{enumerate}
\end{lemma}


In fact, Lemma \ref{thm: abstract_robust_integral_stoch_order} holds for any state space. But it is only a reformulation of what is an optimal action for nature. However, under Assumption \ref{ass:abstract_robust_real} it has practical relevance when a simpler sufficient condition for the integral stochastic order $\leq_{\BB}$ is fulfilled. We give three exemplary criteria:
\begin{enumerate}
	\item[1.] Let $\mathcal{Z}$ be a partially ordered space, e.g.\ $\mathcal{Z}=\R^m$, and assume that the transition functions are  increasing in $z$. Then the functions 
	\begin{align}\label{eq:abstract_robust_convex_composition}
	\mathcal{Z} \ni z \mapsto  c_n(x,a,T_n(x,a,z)) +  v(T_n(x,a,z)), \qquad v \in \BB,\ (x,a) \in D_n 
	\end{align}
	are increasing. Thus, $\Q_1 \leq_{\BB} \Q_2$ is implied by the usual stochastic order of the disturbance distributions $\Q_1^{Z_{n+1}} \leq_{st} \Q_2^{Z_{n+1}}$ and a maximal element of $\Qc_{n+1}$ w.r.t.\ $\leq_{st}$ allows the same conclusion as in Lemma \ref{thm: abstract_robust_integral_stoch_order} b). 
	\item[2.] Let  $\mathcal{Z}$ be a real vector space, assume a convex model (cf.\ Lemma \ref{thm:abstract_robust_convex}) and let the transition functions $T_n$ additionally be convex in $z$. Now, Lemma \ref{thm:abstract_robust_convex} yields that the functions \eqref{eq:abstract_robust_convex_composition} are convex as compositions of increasing convex and convex mappings. Consequently, $\leq_{\BB}$ is implied by the convex order $\leq_{cx}$ of the disturbance distributions $\Q^{Z_{n+1}}$.
\end{enumerate}

\textbf{Convex order on the set of densities.} Since the probability measures in $\Qc_{n+1}$ are absolutely continuous with respect to the reference probability measure $\P_{n+1}$, we can alternatively consider the set of densities $\Qc_{n+1}^d$ (see \eqref{eq:densities}).
In general, one has to take care both of the marginal distribution of the density and the dependence structure with the random cost when searching for a maximizing density of the Bellman equation
\[ \inf_{a \in D_n(x)} \sup_{Y \in \Qc_{n+1}^d} \E\Big[ \Big(c_n\big(x,a,T_n(x,a,Z)\big) +   J_{n+1}\big(T_n(x,a,Z)\big)\Big) Y \Big]. \]
However, if $\Qc_{n+1}^d$ is sufficiently rich, the maximization reduces to comparing marginal distributions.

\begin{definition}
	The set of densities $\Qc_{n}^d$ is called \emph{law invariant}, if for $Y_1 \in \Qc_{n}^d$ every $Y_2 \in L^q(\Omega_n,\A_n,\P_n)$ with $Y_2 \sim Y_1$ is in $\Qc_{n}^d$, too.
\end{definition}

\begin{lemma}\label{thm:abstract_robust_ambiguity_comonotonic}
	Let Assumptions \ref{ass:abstract_robust_finite}, \ref{ass:abstract_robust_real} be satisfied. If $\Qc_{n+1}^d$ is law invariant, 
	\begin{equation}\label{eq:expectation}
	 \sup_{Y \in \Qc_{n+1}^d} \E\Big[ \Big(c_n\big(x,a,T_n(x,a,Z_{n+1})\big) +   v\big(T_n(x,a,Z_{n+1})\big)\Big) Y \Big], \qquad (x,a) \in D_n, \ v \in \BB, \end{equation}	 
	is not changed by restricting the maximization to densities which are comonotonic to the random variable $T_n(x,a,Z_{n+1})$.
\end{lemma}

\begin{proof}
By Lemma \ref{thm:abstract_robust_Lpbound} $c_n\big(x,a,T_n(x,a,Z_{n+1})\big) +   v\big(T_n(x,a,Z_{n+1})\big)$ are in $L^p(\Omega_{n+1},\A_{n+1},\P_{n+1})$ for all $(x,a) \in D_n$. Thus the expectation \eqref{eq:expectation}
	exists for all $Y \in  L^q( \mathcal{Z},  \mathfrak{Z}, \P^{Z_{n+1}} )$. Note that a product of r.v.\ with fixed margins is maximized in expectation when the r.v.\ are chosen comonotonic. Due to the law invariance of $\Qc_{n+1}^d$ we can find for every $Y \in \Qc_{n+1}^d$ some $Y' \in \Qc_{n+1}^d$ comonotonic to $c_n\big(x,a,T_n(x,a,Z_{n+1})\big) +   v\big(T_n(x,a,Z_{n+1})\big)$ such that $Y' \sim Y$ and \eqref{eq:expectation}
is maximal with $Y'$.	
	Since, the function $\R \ni x' \mapsto c_n\big(x,a,x'\big) +  v\big(x'\big)$ is increasing, this is the same as requiring comonotonicity to $T_n(x,a,Z_{n+1})$.
\end{proof}

For the comparison of marginal distributions one would naturally think of stochastic orders. Here, the convex order yields a sufficient criterion for optimality. In order to obtain a connection to risk measures, let us introduce the following notations:

\begin{definition}
	Let $F_X$ be the distribution function of a real-valued random variable $X$.
	\begin{itemize}
		\item[a)]  The \emph{(lower) quantile function} of $X$ is the left-continuous generalized inverse of $F_X$
		\[ \q_X(\alpha):=  q^-_X(\alpha):= \inf \{ x \in \R: \ F_X(x) \geq \alpha \}, \qquad \alpha \in (0,1). \]
		\item[b)] The \emph{upper quantile function} of $X$ is the right-continuous generalized inverse  of $F_X$
		\[  q^+_X(\alpha):= \inf \{ x \in \R: \ F_X(x) > \alpha \}, \qquad \alpha \in (0,1). \]
	\end{itemize}
\end{definition}

\begin{lemma}[{\cite[2.1]{Rueschendorf2009}}]\label{thm:A:distribution-fransform}
	For any random variable $X$ on an atomless probability space there exists a random variable $U_X \sim \U(0,1)$ such that 
	\begin{align*}\label{eq:distribution-transform}
	q_X^-(U_X) = q_X^+(U_X) = X \quad \P\text{-a.s.}
	\end{align*}
	The random variable $U_X$ is referred to as \emph{(generalized) distributional transform} of $X$. Note that if $h:\R\to\R$ is increasing and left-continuous, then $X$ and $h(X)$ have the same distributional transform.
\end{lemma}

\begin{definition}
Let  $\phi:[0,1] \to \R_+$ be increasing and right-continuous with $\int_0^1 \phi(u)\dif u=1$. 	Functionals $\rho_{\phi} : L^p \to \bar{\R}$
\begin{align*} 
	\rho_{\phi}(X):= \int_0^1 q_X(u) \phi(u) \dif u.
	\end{align*}
	are called  \emph{spectral risk measures} and $\phi$ is called \emph{spectrum.}
\end{definition}

When we choose the spectrum $\phi(u)=\frac{1}{1-\alpha} 1_{[\alpha,1]}(u)$  then we obtain the \emph{Expected Shortfall.} Note that every spectral risk measure is also coherent, i.e.\ monotone, translation-invariant, positive homogeneous and subadditive.

 In what follows we assume a non-atomic probability space.

\begin{lemma}\label{thm:abstract_robust_ambiguity_density}
	Let Assumptions \ref{ass:abstract_robust_finite}, \ref{ass:abstract_robust_real} be satisfied, $\Qc_{n+1}^d$ be law invariant and suppose there exists a maximal element $Y^*_{n+1}$ of $\Qc_{n+1}^d$ w.r.t.\ the convex order $\leq_{cx}$.
	\begin{itemize}
		\item[a)] Then 
		\begin{align*}
		\rho_\phi(X)= \sup_{Y \in \Qc_{n+1}^d} \E[X Y], \qquad X \in L^p(\Omega_{n+1},\A_{n+1},\P_{n+1}),
		\end{align*}
		defines a spectral risk measure with spectrum $\phi_{n+1}(u):=q^+_{Y^*_{n+1}}(u), \ u \in [0,1]$. In this case, $\gamma=(\gamma_0,\dots,\gamma_{N-1})$ with $\gamma_n(x,a)= \phi_{n+1}(U_{T_n(x,a,Z_{n+1})})$ is an optimal strategy of nature in  \eqref{eq:opt_crit_abstract_robust_finite}. Here $U_{T_n(x,a,Z_{n+1})}$ denotes the distributional transform of $T_n(x,a,Z_{n+1})$.
		\item[b)] If additionally, the disturbance space is the real line $\mathcal{Z}=\R$ and the transition function $T_n$ is increasing and lower semicontinuous in $z$, $\gamma=(\gamma_0,\dots,\gamma_{N-1})$ with $\gamma_n \equiv \phi_{n+1}(U_{Z_{n+1}})$ defines a constant optimal action of nature. That is, \eqref{eq:opt_crit_abstract_robust_finite}  can be reformulated to a risk-neutral MDP with probability measures $\dif \Q_{n+1} = \phi_{n+1}(U_{Z_{n+1}}) \dif \P_{n+1}$.
	\end{itemize}
\end{lemma}

\begin{proof}
	\begin{enumerate}
		\item It holds $Y^*_{n+1}= q^+_{Y^*_{n+1}}(U_{Y^*_{n+1}})\ \P$-a.s.\ by Lemma  \ref{thm:A:distribution-fransform}. Therefore,
		\[ \Qc_{n+1}^d \subseteq \left\{ Y \in L^q(\Omega_{n+1},\A_{n+1},\P_{n+1}): \ Y \leq_{cx} \phi_{n+1}(U), \ U \sim \U(0,1) \right\}, \]
		and the random variables $\phi_{n+1}(U), \ U \sim \U(0,1)$ are contained in both sets due to law invariance. $\rho_\phi$ indeed defines a spectral risk measure (see Proposition \ref{thm:spectral_rm-dual})		
		 and $\phi_{n+1}(\tilde U)$ is an optimal action of nature at time $n$ given $(x,a)$, where $\tilde U$ is the distributional transform of the random variable
		$$c_n\big(x,a,T_n(x,a,Z_{n+1})\big) +   J_{n+1}\big(T_n(x,a,Z_{n+1})\big).$$
		Since the function $\R \ni x' \mapsto c_n\big(x,a,x'\big) +   v\big(x'\big)$ is increasing and lower semicontinuous, i.e.\ left continuous, it follows with Lemma \ref{thm:A:distribution-fransform} that $\tilde U = U_{T_n(x,a,Z_{n+1})}$. 
		\item Under the additional assumptions we have that $U_{T_n(x,a,Z_{n+1})} = U_{Z_{n+1}}$. \qedhere
	\end{enumerate}
\end{proof} 

Recall that the probability space under consideration is the product space.
Under the assumptions of Lemma \ref{thm:abstract_robust_ambiguity_density} b) we can replace the probability measure $\P$ by
\[ \widehat \Q:=  \bigotimes_{n=1}^{N-1} \Q^*_n, \qquad \dif \Q^*_n= \phi_n(U_{Z_{n}}) \dif \P_n \] 
and the optimization problem \eqref{eq:opt_crit_abstract_robust_finite} can  equivalently be written as 
\begin{align}\label{eq:abstract_robust_ambiguity_risk-neutral}
\inf_{\pi \in \Pi^M} \E_x^{\widehat \Q} \left[ \sum_{n=0}^{N-1}  c_n(X_n,d_n(X_n),X_{n+1}) +c_N(X_N)\right].
\end{align}
 With the reversed argumentation of Lemma \ref{thm:abstract_robust_ambiguity_density} a robust formulation of \eqref{eq:abstract_robust_ambiguity_risk-neutral} is given by
\begin{align}\label{eq:abstract_robust_ambiguity_new_robust}
\inf_{\pi \in \Pi^M} \sup_{\Q \in \Qf} \E_x^{\Q} \left[ \sum_{n=0}^{N-1}  c_n(X_n,d_n(X_n),X_{n+1}) +c_N(X_N)\right]
\end{align} 
with
\[ \Qf= \left\{ \bigotimes_{n=1}^{N-1} \Q_n: \ \dif \Q_n = Y_n \dif \P_n, \  Y_n \in L^q(\Omega_n,\A_n,\P_n), \ Y_n \leq_{cx} \phi_n(U_n), \ U_n \sim \U(0,1)  \right\}. \]
The $Y_n,$ are indeed densities. Now, \eqref{eq:abstract_robust_ambiguity_new_robust} can be interpreted as the minimization of a coherent risk measure $\rho$
\begin{align}\label{eq:eq:abstract_robust_ambiguity_new_robust_rm}
\inf_{\pi \in \Pi^M} \rho\Big( \sum_{n=0}^{N-1}  c_n(X_n,d_n(X_n),X_{n+1}) +c_N(X_N)\Big)
\end{align}
and the problem can be solved with the value iteration
$$J_n(x) = \inf_{a\in D_n(x)} \rho_n \Big( c_n(x,a,T_n(x,a,Z_{n+1}))+J_{n+1}(T_n(x,a,Z_{n+1}))\Big)$$
where $\rho_n(X) = \sup_{Y \in \Qc_{n+1}^d} \E[X Y]$ and $J_0(x)$ gives the minimal value \eqref{eq:abstract_robust_ambiguity_new_robust}.

\begin{remark}
Some authors already considered the problem of optimizing risk measures of dynamic decision processes. For example  \cite{ruszczynski2010risk} considered Markov risk measures for finite and discounted infinite horizon models. In the final section he briefly discusses the relation to min-max problems where player 2 chooses the measure. 
\cite{shen2013risk} treat similar problems (also with average cost)  with different properties of the risk maps. More specific applications (dividend problem and economic growth) with the recursive entropic risk measures are treated in \cite{BaeuerleJaskiewicz2017, BaeuerleJaskiewicz2018}. In \cite{chow2015risk} the authors treat the  dynamic average-value at risk as a min-max problem. For this risk measure there are also alternative ways for the solution, see e.g. \cite{BaeuerleOtt2011}, 
\end{remark}

\section{Application}
\subsection{LQ Problem}
We consider here so-called linear-quadratic (LQ) problems. The state and action space are $E=\R$ and $A=\R^d$. Let $U_1,\dots, U_n$, $V_1,\dots,V_N$ be $\R$- and $\R^{ d}$-valued random vectors, respectively, and $W_1,\dots,W_N$ be random variables with values in $\R$. The random elements $\{Z_n=(U_n,V_n,W_n)\}_{1\leq n \leq N}$ are independent and the $n$-th element is defined on $(\Omega_n,\A_n,\P_n)$. It is supposed that the disturbances $\{Z_n\}_{1\leq n \leq N}$ have finite $2p$-th moments, $p \geq 1$.   The transition function is given by
\[ T_n(x,a,Z_{n+1}) = U_{n+1}x + V_{n+1}^\top a + W_{n+1} \]
for $n=0,\dots,N-1$. Furthermore, let there be deterministic positive constants  $Q_0,\dots,Q_N \in \R_+$ and deterministic positive definite symmetric matrices $R_0,\dots,R_{N-1} \in \R^{d \times d}$. The one-stage cost functions are
\[ c_n(x,a,x') = c_n(x,a)=x^2 Q_n  + a^\top R_n a \]
and the terminal cost function is $c_N(x)=x^2 Q_N $. Hence, the optimization problem under consideration is
\begin{align}\label{eq:opt_crit_LQ}
\inf_{\pi \in \Pi^R} \sup_{\gamma \in \Gamma} \ \E_{0x}^{\pi \gamma} \left[ \sum_{k=0}^{N-1} X_k^2 Q_k + A_k^\top R_k A_k + X_N^2 Q_N  \right].
\end{align}
Policy values and value functions are defined in the usual way. 
For different formulations of robust LQ problems, see e.g. \cite{hansen1999five}.

Since  $Q_n$ and the matrices $R_n$ are positive semidefinite, $\ubar b \equiv 0$ is a lower bounding function and the one-stage costs are at least quasi-integrable. In the sequel, we will determine the value functions and optimal policy by elementary calculations and will show that the value functions are convex and therefore continuous. Hence, we can dispense with an upper bounding function and compactness of the action space.

Since the Borel $\sigma$-algebra of a finite dimensional Euclidean space is countably generated, it is no restriction to assume that the probability measures $\P_{n+1}$ are separable. Further, we assume that for $n=0,\dots,N-1$
\begin{itemize}
	\item[(i)] the ambiguity sets $\Qc_{n+1} \subseteq \M^q(\Omega_{n+1},\A_{n+1},\P_{n+1})$ are norm bounded and weak* closed.
	\item[(ii)] it holds $\E^\Q[W_{n+1}]=0$ for all $\Q \in \Qc_{n+1}$. 
	\item[(iii)] $W_{n+1}$ and $(U_{n+1},V_{n+1})$ are independent for all $\Q \in \Qc_{n+1}$, i.e. $\Q=\Q_{W_{n+1}}\otimes \Q_{U_{n+1},V_{n+1}}$. 
\end{itemize}
I.e.\ Assumption \ref{ass:abstract_robust_finite} is satisfied apart from upper bounding. Theorems \ref{thm:abstract_robust_nature_bellman} and \ref{thm:abstract_robust_finite} use the bounding, continuity and compactness assumptions only to prove the existence of optimal decision rules. Thus, we can employ the Bellman equation and restrict the consideration to Markov policies as long as we are able to prove the existence of optimal decision rules on each stage. We proceed backwards.
At stage $N$, no action has to be chosen and the value function is $J_N(x)= x^2 Q_N $.
At stage $N-1$, we have to solve the Bellman equation
\begin{align}\label{eq:LQ_1}
	J_{N-1}(x)&= \inf_{a \in A} \sup_{\Q \in \Qc_{N}} c_{N-1}(x,a) + \E^\Q\left[ J_N(T(x,a,Z_{n+1})) \right]\\
	&= \inf_{a \in A} \sup_{\Q \in \Qc_{N}} x^2 Q_{N-1}  + a^\top R_{N-1} a+ \E^\Q\left[ (U_N x + V_N^\top a + W_N)^2 Q_N \right]\notag\\
	&= \inf_{a \in A} \sup_{\Q \in \Qc_{N}} x^2  Q_{N-1} + a^\top R_{N-1} a + \E^\Q\left[ x^2 U_N^2   +  (V_N^\top a)^2   \right.
\left. 
	+ 2 x U_N  V_N^\top a + W_N^2   \right] Q_N \notag
\end{align}
For the last equality we used  that $\E^\Q\left[ 2 W_N (U_Nx+V_Na)  \right]=0$ by assumption. Since $R_{N-1}$ and $Q_N$ are positive (semi-) definite, the objective function \eqref{eq:LQ_1} is strictly convex in $a$. Moreover, it is linear and especially concave in $\Q$. Finally, $\Qc_N$ is weak* compact by the Theorem of Banach-Alaoglu \cite[6.21]{AliprantisBorder2006} the  objective function \eqref{eq:LQ_1} is continuous in $\Q$ by definition of the weak* topology since the integrand is in $L^p(\Omega_{n+1},\A_{n+1},\P_{n+1})$. Thus, the requirements of Sion's Minimax Theorem \ref{thm:A:minimax} b) are satisfied and we can interchange infimum and supremum in \eqref{eq:LQ_1}, i.e.\
\begin{align}\label{eq:LQ_2}
	J_{N-1}(x) 
	&= \sup_{\Q \in \Qc_{N}} \inf_{a \in A} x^2  Q_{N-1}  + a^\top R_{N-1} a +  x^2\E^\Q[ U_N^2] Q_N  + a^\top \E^\Q[V_N^\top V_N] a Q_N \notag\\ 
	&\phantom{= \inf_{a \in A} \sup_{\Q \in \Qc_{N}}}\ + 2 x  Q_N \E^\Q[U_N V_N^\top] a + \E^\Q[W_N^2] Q_N
\end{align}

In order to solve the inner minimization problem it suffices due to strict convexity and smoothness to determine the unique zero of the gradient of the objective function.
\begin{align*}
	& 0 = 2R_{N-1} a +2\E^\Q[V_N^\top V_N] a Q_N + 2 x \E^\Q[V_N^\top U_N] Q_N\\
	\Longleftrightarrow \quad & a= - \big(R_{N-1} + \E^\Q[V_N^\top V_N]Q_N \big)^{-1} \E^\Q[V_N^\top  U_N]Q_N x. 
\end{align*}
Note that the matrix $ \big(R_{N-1} + \E^\Q[V_N^\top V_N] \big)Q_N$ is positive definite and therefore invertible due to the positive (semi-) definiteness of $R_{N-1}$ and $Q_N$. 
Inserting in \eqref{eq:LQ_2} gives
\begin{align}\label{eq:LQ_3}
	J_{N-1}(x) 
	&=\sup_{\Q \in \Qc_{N}} \E^\Q[W_N^2] Q_N  + x^2 \Big( Q_{N-1} +\E^\Q[ U_N^2] Q_N \notag\\ 
	&\phantom{=\sup_{\Q \in \Qc_{N}}}\ - \E^\Q[U_N V_N] \Big(R_{N-1} + \E^\Q[V_N^\top  V_N] Q_N \Big)^{-1} \E^\Q[V_N^\top  U_N]Q_N^2 \Big)\notag\\
	&=\sup_{\Q \in \Qc_{N}} \E^\Q[W_N^2] Q_N  + x^2 K_{N-1}^\Q = 
	\sup_{\Q_{W_N}} \E^\Q[W_N^2] Q_N  + x^2 \sup_{\Q_{U_N,V_N}}  K_{N-1}^\Q 
\end{align}
where
\begin{align*}
	K_{N-1}^\Q &= Q_{N-1} +\E^\Q[ U_N^2] Q_N  - \E^\Q[U_N^\top  V_N] \Big(R_{N-1} + \E^\Q[V_N^\top  V_N] Q_N \Big)^{-1} \E^\Q[V_N^\top  U_N] Q_N^2.
\end{align*}
Since $J_n(x)\ge 0$ for all $x\in\R$ we must have $K_{N-1}^\Q>0$ and since $\Qc_{N}$ is weak* compact and $\Q \mapsto \E^\Q[W_N^2] Q_N  + x^2 K_{N-1}^\Q$ weak* continuous, there exists an optimal measure $\gamma_{N-1}(x)=\Q^*_N=\Q^*_{W_N}\otimes \Q^*_{U_N,V_N}$ which is independent of $x$ due to our independence assumption (iii). Since $J_{N-1}$ is again quadratic we obtain when we define
\begin{align*}
	K_{n-1}^\Q &= Q_{n-1} +K_n^\Q \Big(\E^\Q[ U_n^2]   - \E^\Q[U_n^\top  V_n] \Big(R_{n-1}/K_n^\Q + \E^\Q[V_n^\top  V_n]  \Big)^{-1} \E^\Q[V_n^\top  U_n] \Big)\\
	L_{n-1}^\Q &= - \Big(R_{n-1}/K_n^\Q + \E^\Q[V_n^\top  V_n] \Big)^{-1} \E^\Q[V_n^\top  U_n] 
\end{align*}
that $\gamma^*_n(x)\equiv \gamma^*_n= \argmax_{\Q_{W_n}} \E^\Q[W_n^2]\otimes \argmax_{\Q_{U_n,V_n}}  K_n^\Q $ and $d_n^*(x) =L_n^{\gamma^*_n} x$. Since the third term of $K_{n-1}^\Q$ is a quadratic form nature should choose $\Q$ such that $\E^\Q[V_n^\top  U_n] =0$ if possible and maximize the second moment of $U_n$. If this is possible, we obtain  $d_n^*(x)=0$, i.e. the controller will not control the system. In any case we see here the optimal choice of nature $\gamma^*_n(x)$ does not depend on $x$. 

When we specialize the situation to $d=1$ we obtain
\begin{align*}
	K_{n-1}^\Q &= Q_{n-1} +\Big( \E^\Q[ U_n^2]   -\frac{(\E^\Q[U_n V_n]^2}{\frac{R_{n-1}}{K_n^\Q}+\E^\Q[V_n^2 ] } \Big)  K_n^\Q.
\end{align*}
Thus, nature has to maximize the expression in brackets. For a moment we skip the index $n$. Let us further assume that $(U,V)$ is jointly normally distributed with expectations $\mu_U$ and $\mu_V$ and variances $\sigma^2_U$ and $\sigma^2_V$ and covariance $\sigma_{U,V}$ and that all these parameters may be elements of compact intervals. Then the expression in brackets reduces to
$$\sigma_U^2 +\mu_U^2 - \frac{(\sigma_{UV}+\mu_U\mu_V)^2}{\frac{R}{K^\Q}+\sigma_V^2 +\mu_V^2 }.$$
We see immediately that both $\sigma_V^2$ and $\sigma_U^2$ have to be chosen as maximal possible value. The remaining three parameters $\mu_U, \mu_V$ and $\sigma_{UV}$ have to minimize the fraction. If it is possible to choose them such that $\mu_U$ is maximal and $ \sigma_{UV}+\mu_U\mu_V=0$ this would be optimal.

\subsection{Managing regenerative energy}
The second example is the management of a joint wind and storage facility. Before each period the owner of the wind turbine has to announce the amount $a$ of energy she wants to produce. If there is enough wind in the next period she receives the reward $P a$ where we assume that $P>0$ is the fixed price for energy. Additional energy which may have been produced will be stored in a battery with capacity $K>0$. If there is not enough wind in the next period, the storage device will be used to cover the shortage. In case this is still not enough, the remaining shortage will be penalized by a proportional cost rate $c>0$ (see \cite{anderson2015optimal} for further background). We consider a robust version here, i.e. the distribution $\Q$ of the produced wind energy varies in a set $\Qc$ with bounded support $[0,B]$. The state is the amount of energy in the battery, hence $E=[0,K]$. Further $A=[0,B]=D(x)$ and the action is the amount of energy which is bid.   We obtain the following Bellman equation:
\begin{eqnarray*}
J_n(x) &=& \inf_{a\in D(x)} \sup_{\Q\in \Qc} \Big\{ \int_a^B -aP +J_{n+1}((x+z-a)\wedge K)\Q(dz)\\
&& + \int_0^a -(a \wedge (z+x)) P +(a-z-x)^+ c +J_{n+1}((x+z-a)^+) \Q(dz)\Big\}
\\
&=&  \inf_{a\in D(x)} \sup_{\Q\in \Qc} \Big\{ -aP+\int_a^B J_{n+1}((x+z-a)\wedge K)\Q(dz)\\
&& + \int_0^a (P+c)(x+z-a)^-  +J_{n+1}((x+z-a)^+) \Q(dz)\Big\}
\end{eqnarray*}
From the last representation we see that 
$$
T(x,a,z)=  \left\{ \begin{array}{ll}
(x+z-a)\wedge K), & z\ge a\\
(x+z-a)^+,& 0\le z\le a
\end{array}\right.$$
and
$$
c(x,a,T(x,a,z))= - aP + \left\{ \begin{array}{ll}
0, & z\ge a\\
(P+c)(x+z-a)^-,& 0\le z\le a
\end{array}\right.$$
We see that $D(x)$ is compact and $x\mapsto D(x)$ is lower semicontinuous. $T$ is continuous and $c$ is continuous and bounded. It is easy to see that $J_n(x)$ is decreasing in $x$. Suppose $\Q$ has a minimal element $\Q^*$ w.r.t.\ $\le_{st}$. Then we can omit the $\sup_{\Q\in \Qc} $ and replace $\Q$ by $\Q^*$. The remaining Bellman equation is then a standard MDP.

\section{Appendix}
\subsection{Additional Proofs}
Proof of Lemma \ref{thm:ambiguity_sep_metrizable}:

	Recall that we identify $\Qc_n$ with the set of the corresponding densities $\Qc_n^d$. The closure $\overline \Qc_n^d$ of $\Qc_n^d$ remains norm bounded. This can be seen as follows: Let $ X \in \overline \Qc_n^d$. Then there exists a net $\{X_\alpha\}_{\alpha \in I} \subseteq \Qc_n^d$ such that $X_\alpha \overset{w^*}{\to} X$. Hence,
	\[ \E[X_\alpha Y] \to \E[XY] \qquad \text{for all } Y \in L^p(  \Omega_n,  \A_n, \P_n ) \text{ with } \|Y\|_{L^p}=1. \]
	By H\"older's inequality we have for all $\alpha \in I$
	\[ \left| \E[X_\alpha Y] \right| \leq \E|X_\alpha Y| \leq \|X_\alpha\|_{L^q} \|Y\|_{L^p} = \|X_\alpha\|_{L^q} \leq K.  \]
	Thus, $|\E[XY]| \leq K$. Finally, due to duality it follows
	\[ \|X\|_{L^q} = \sup_{\|Y\|_{L^p}=1 } |\E[XY]| \leq K. \]
	The separability of the probability measure $\Pop_n$ makes $L^p(\Omega_n,\A_n,\Pop_n)$ a separable Banach space. 
Consequently, the weak* topology is metrizable on the norm bounded set $\overline \Qc_n^d$ \cite[p. 157]{Morrison2001}. The trace topology on the subset $\Qc_n^d \subseteq \overline \Qc_n^d$ coincides with the topology induced by the restriction of the metric \cite[4.4.1]{OSearcoid2007}, i.e.\ $\Qc_n^d$ is metrizable, too.
	
	Since $\overline \Qc_n^d$ is norm bounded and weak* closed, the Theorem of Banach-Alaoglu \cite[6.21]{AliprantisBorder2006}  yields that it is weak* compact. As a compact metrizable space $\overline \Qc_n^d$ is complete \cite[3.28]{AliprantisBorder2006} and also separable \cite[3.26, 3.28]{AliprantisBorder2006}. Hence, $\overline \Qc_n^d$ is a Borel Space. The set of densities $\Qc_n^d$ is also separable as a subspace of a separable metrizable space \cite[3.5]{AliprantisBorder2006}. \hfill $\Box$\\
	
Proof of Lemma \ref{thm:abstract_robust_kernel}:

\begin{proof}
	By definition, $\gamma_n(\cdot|h_n,a_n)$ is a probability measure for every $(h_n, a_n)\in \H_n\times A$. Now fix $B \in \A_{n+1}$. The mapping $\delta: \Qc_{n+1} \to [0,1], \ \delta(B) = \E^\Q[1_B]$ is weak*-continuous since $1_B \in L^p( \Omega_{n+1},  \A_{n+1}, \P_{n+1})$ and hence Borel measurable. Therefore, 
	\[ \H_n\times A \ni (h_n,a_n) \mapsto \gamma_n(B|h_n,a_n) = \delta \circ \gamma_n(h_n,a_n) \]
	is measurable as a composition of measurable maps. 
\end{proof}

\subsection{Additional Statements}
\begin{lemma}\label{thm:abstract_robust_Lpbound}
	Let $v \in \BB_b$ and $n \in \{0,\dots,N-1\}$. Under Assumption \ref{ass:abstract_robust_finite} (ii) each sequence of random variables
	\[ C_k= c_n\big(x_k,a_k,T_n(x_k,a_k,Z_{n+1})\big) + v\big(T_n(x_k,a_k,Z_{n+1})\big) \] 
	induced by a convergent sequence $\{(x_k,a_k)\}_{k \in \N}$ in $D_n$ has an $L^p$-bound $\bar C$, i.e.\ $|C_k| \leq \bar C \in L^p(\Omega_{n+1},\A_{n+1},\P_{n+1})$ for all $k \in \N$. 
\end{lemma}

\begin{proof}
	There exists a constant $\lambda \in \R_+$ such that $|v|\leq \lambda b$. Since $D_n$ is closed, the limit point $(x_0,a_0)$ of $\{(x_k,a_k)\}_{k \in \N}$ lies in $D_n$. Let $\epsilon>0$ be the constant from Assumption \ref{ass:abstract_robust_finite} (ii) corresponding to $(x_0,a_0)$. Since the sequence is convergent, there exists $m \in \N$ such that $(x_k,a_k) \in B_\epsilon(x_0,a_0) \cap D_n$ for all $k > m$. For the finite number of points $(x_0,a_0),  (x_1,a_1), \dots, (x_m,a_m)$ there exist bounding functions $\Theta_{n,1}^{x_i,a_i}, \Theta_{n,2}^{x_i,a_i}$ by Assumption \ref{ass:abstract_robust_finite} (iii). Thus, the random variable 
	\[ \bar C= \max_{i=0,\dots,m} \Big(\Theta_{n,1}^{x_i,a_i}(Z) + \lambda \Theta_{n,2}^{x_i,a_i}(Z)  \Big) \]
	is an $L^p$-bound as desired.
\end{proof}

\begin{theorem}[{\cite[4.1, 4.2]{Sion1958}}] \phantomsection \label{thm:A:minimax}
	\begin{itemize}
		\item[a)] Let $X$ be any set, $Y$ compact and $f:X \times Y \to \bar \R$ concave-convex-like and lower semicontinuous in the second argument, then
		\[ \sup_{x \in X} \inf_{y \in Y} f(x,y) = \inf_{y \in Y} \sup_{x \in X} f(x,y).  \]
		\item[b)] Let $X$ be compact, $Y$ any set and  $f:X \times Y \to \bar \R$ concave-convex-like and upper semicontinuous in the first argument, then
		\[ \sup_{x \in X} \inf_{y \in Y} f(x,y) = \inf_{y \in Y} \sup_{x \in X} f(x,y).  \]
	\end{itemize}
\end{theorem}

Spectral risk measures posses a specific dual representation becomes. The following statement can be found in \cite{Pichler2015}.

\begin{proposition}\label{thm:spectral_rm-dual}
	A spectral risk measure $\rho_\phi: L^p \to \R$ with spectrum $\phi \in L^q$ can be represented as
	\begin{itemize}
		\item[(i)] $ \rho_\phi(X)= \sup_{ U \sim \U(0,1) } \E[X \phi(U)]. $
		\item[(ii)] $ \rho_\phi(X)= \sup\Big\{\E[X Y]: Y \in L^q,  \ Y \leq_{cx} \phi(U), \ U \sim \U(0,1) \Big\}.  $
	\end{itemize}
	The suprema are attained and the maximizer is given by $\phi(U_X)$, where $U_X$ is the generalized distributional transform of $X$.
\end{proposition}

\bibliographystyle{amsplain}
\bibliography{Literature_Dissertation_Glauner}

\end{document}